\documentclass[11pt,leqno]{article}

\usepackage{cite}

\usepackage{amsthm,amsfonts,amssymb,amsmath,oldgerm,mathtools}

\usepackage{fullpage}
\usepackage{graphicx}
\usepackage{float}

\usepackage{mathrsfs}
\usepackage{xcolor}

\numberwithin{equation}{section}

\newcommand{\realpart}[1]{\operatorname{\rm Re}\!\left(#1\right)}
\newcommand{\impart}[1]{\operatorname{\rm Im}\!\left(#1\right)}

\newcommand{\eq}[2]{\begin{equation}\begin{split}#1\end{split}\label{#2}\end{equation}}

\newcommand{\RM}{{\mathbb{R}}}
\newcommand{\CM}{{\mathbb{C}}}

\newtheorem{theorem}{Theorem}[section]
\newtheorem{proposition}[theorem]{Proposition}
\newtheorem{corollary}[theorem]{Corollary}
\newtheorem{lemma}[theorem]{Lemma}
\newtheorem{remark}[theorem]{Remark}
\theoremstyle{definition}

\allowdisplaybreaks[3]

\newcommand{\red}[1]{{{#1}}}

\newcommand{\ri}{\mathrm{i}}            

\title{Linear Asymptotic Stability of the Smooth 1-Solitons for the Degasperis-Procesi Equation}

\author{
Simon Deng\thanks{{{Department of Mathematics, University of Kansas, 1460 Jayhawk Boulevard, 
Lawrence, KS 66045, USA; simondeng@ku.edu}}}\and Mathew~A.~Johnson\thanks{Department of Mathematics, University of Kansas, 1460 Jayhawk Boulevard, 
Lawrence, KS 66045, USA; matjohn@ku.edu}\and 
St\'ephane Lafortune\thanks{Department of Mathematics, College of Charleston, Charleston, SC 29401, USA; lafortunes@cofc.edu}
}

\date{\today}

\begin{document}

\maketitle

\begin{abstract}
In this paper, we study the asymptotic stability of smooth 1-solitons in the Degasperis-Procesi (DP) equation.  Such solutions
necessarily exist on a non-zero background, and their spectral and orbital stability has previously been verified by 
Li, Liu \& Wu and by Lafortune \& Pelinovsky.  
Using the complete integrability of the DP equation to establish the strong spectral stability of smooth
solitary waves, namely that the origin is the only eigenvalue of the associated linearized
operator acting on $L^2(\RM)$ and that, moreover, in appropriate exponentially weighted spaces
the non-zero spectrum for the linearized operator admits a spectral gap away from the imaginary axis.
This spectral gap result {{is then}} upgraded to an exponential decay estimate on the semigroup associated with the linearized operator,
establishing a linear asymptotic stability result in exponentially weighted spaces.  Finally, we outline analytical challenges with
extending our result to the nonlinear level.
\end{abstract}
	

  \tableofcontents

\section{Introduction}

In this paper, we consider the asymptotic dynamics near smooth solitary wave solutions of the 
Degasperis-Procesi (DP) equation
\begin{equation}\label{DP} 
u_t - u_{txx} = 3u_x u_{xx} - 4uu_x + uu_{xxx}.
\end{equation}
The DP equation was derived in  \cite{dp} when Degasperis and Perocesi considered a family of equations to which they applied asymptotic integrability conditions up to the third order. In addition to the Kortweg-de Vries (KdV) and DP equations, in this family they additionally found the well-known Camassa-Holm equation
\begin{equation} \label{CH}
u_t - u_{txx} = 2u_x u_{xx} - 3uu_x + uu_{xxx},
\end{equation}
which was introduced in \cite{Cam,Cam2} as a model of strongly dispersive, unidirectional  shallow water waves.
Since then the DP equation has found analogous applications in the dynamics of surface water
waves with the same asymptotic accuracy as the CH equation \cite{constantin}. 
Further, both the CH and DP equations are known to be completely integrable
via the Inverse Scattering Transform, possess infinitely many conservation laws, are bi-Hamiltonian, and admit both smooth and non-smooth solutions such as 
peakons (i.e. continuous solutions with discontinuous yet bounded derivatives) and multi-peakons \cite{Cam,Cam2,honep,Constantini,Constantin2010,dhh,Ma05}.

Despite their similarities, the CH and DP equations have significant mathematical differences.  While these differences are numerous, a significant one
that is particularly relevant in the current work is that the isospectral problem in the Lax pair for the CH is a second-order equation (hence self-adjoint
in an appropriate weighted space), while the associated isospectral problem for the DP equation is a third-order, non-self adjoint equation.  
Additionally, unlike the CH equation, the DP equation admits  discontinuous solutions \cite{Giuseppe1,Giuseppe2,Lundmark1} and, furthermore,  
peakon-antipeakon collisions lead to  ``shock-peakons'' \cite{Lundmark2007} that have jump singularities in $u$.

In this work, we are interested in the dynamics of smooth, solitary wave solutions of the DP equation when subjected to small perturbations.  
The spectral and nonlinear orbital stability of smooth solitary waves of the DP equation was obtained in \cite{Li2019,Li2024}.  Further, an energy stability criterion
for the orbital stability of such solutions (relying on a different Hamiltonian structure than in \cite{Li2024}) was derived for an entire family of 
generalized CH/DP equations in \cite{LP22smooth}, and a followup work \cite{LTC23} demonstrated that this criterion holds for all smooth solitary waves.  

\
\begin{center}
\emph{The purpose of this work is to study the asymptotic stability of smooth solitary waves of the DP equation.}
\end{center}
\

In the literature, there seems to be two main approaches to studying the asymptotic stability of solitary waves for nonlinear dispersive equations.  One dates back
to Pego \& Weinstein's seminal study of the asymptotic stability of solitary waves of the generalized KdV equation \cite{Pego1997}.  This method,
which seemingly has roots in classical dynamical systems theory, relies on a precise spectral and linear analysis of the associated 
linearized evolution equation about a solitary wave posed on exponentially weighted spaces, and then using 
nonlinear iteration and symmetry modulation techniques to establish nonlinear asymptotic stability on these weighted spaces. 
Pego \& Weinstein's methodology has since been applied in a number of settings: see \cite{Pego94,Simpson2008,Miller1996,Pego1997,Sattinger76},
for instance\footnote{We also note \cite{M01} where the author considers instead polynomial weighted spaces.}.  
The other method, {\red{which seemingly}} has its roots in more classical mathematical analysis, aims to establish a so-called nonlinear Liouville property
around solitary waves.  The Liouville property essentially says (stated here very formally for the ease of exposition) that solutions that are near a solitary wave
and exhibit appropriate exponential decay properties must themselves be solitary waves.  This second methodology has proven to be extremely
powerful and is capable of establishing asymptotic stability in low regularity  settings (even $L^2(\RM)$).  The works are numerous, but see, for instance,
\cite{BGS15,CMPS16,KM09,Mol18,Molinet2019} and references therein.

In this work, we follow the dynamical systems approach of Pego \& Weinstein \cite{Pego1997}.  However,
we note {{that}} due to a subtle issue coming from the quasilinear nature of the DP equation that we are unable to obtain via this methodology
the asymptotic stability of DP solitary waves at the nonlinear level.  This issue will be discussed in Section \ref{S:nonlinear_try}, but effectively has to do
with a loss of regularity in a nonlinear iteration scheme that can seemingly not be regained through linear smoothing estimates or by
rescalings of the problem.  Nevertheless, a key step in the method of Pego \& Weinstein is to establish an appropriate asymptotic 
stability theory in exponentially weighted spaces at the linear level that accounts for drifts along the continuous symmetries of the governing PDE.  
This work itself is highly nontrivial and, while not a full nonlinear result, still sheds light on the local dynamics about solitary waves
of DP equation.

\
\begin{center}
\emph{In this work, we will study the linear asymptotic stability of smooth 1-soliton solutions of the DP equation in exponentially weighted spaces.}
\end{center}
\

The topic of asymptotic stability of smooth solutions in Camassa-Holm type equations is very much in its infancy.
In the very recent work \cite{CLLW24}, the authors have obtained the asymptotic stability of smooth N-Solitons for the CH equation \eqref{CH}.  
Their argument relies on the bi-Hamiltonian structure, the complete integrability of the CH equation, 
properties of the recursion operator associated to the CH hierarchy, and the idea
of a ``rigidity property" (similar to a nonlinear Liouville property)
introduced by by Molinet in \cite{Mol18,Molinet2019} in their asymptotic stability analysis of CH and DP peakon equations.  Additionally,
they rely on the fact that a formula for solutions of the linearization in the form of squared eigenfunctions from the Lax pair is known and their completeness has been
established \cite{Constantin2007}.  In our case, the only tool from complete integrability that 
we make use of is a squared eigenfunction connection for which, in the DP case, no formula 
or completeness result is currently known.  Furthermore, as discussed above, the nature of the Lax pairs of the CH and  the DP equations are very different,
with the latter having an isospectral problem that is third-order and non-self adjoint. 
Other than the current contribution, the only other related work in this direction is the also recently submitted work \cite{WLZ25} where they consider
the asymptotic stability of smooth solitary waves to the so-called b-family of equations 
\[
u_t-u_{xxt} = bu_xu_{xx}-(b+1)uu_x+uu_{xxx},
\]
in the case where $b<-1$.  Unlike the completely integrable cases of $b=2$ (CH) and $b=3$ (DP), the b-family is not integrable for $b<-1$.  In \cite{WLZ25},
the authors instead adapt the Martel-Merle framework for generalized KdV equations by combining a nonlinear Liuoville property for appropriately localized solutions along with
a refined spectral analysis of the associated linearized operator.  Notably, however, the work \cite{WLZ25} does not include the DP equation considered here.

As heuristic motivation for our result, we first note that smooth traveling wave solutions to the DP equation only exist on non-zero constant background states, i.e. they have
non-zero asymptotic limits at spatial infinity.  For a given admissible constant background state, the DP equation admits a two-parameter family of smooth
traveling wave solutions (the 1-solitons) which are parameterized by the wave speed $c$ and translation invariance.  
With this in mind, there are two main heuristic mechanisms driving the expected stability\footnote{These are essentially the same
heuristics outlined in \cite{Pego1997}.}.  First, on a fixed background we find that smaller solitary waves 
traveling slower than larger ones: this is quantified precisely in \eqref{e:max} below, which shows that the wave {\red{amplitude}}, 
defined as the difference between the global maximum and its asymptotic end state at spatial infinity,
is a strictly increasing function of the wave speed $c$.  Secondly, the group velocity of linear plane wave solutions of the DP equation (considered on the non-zero backgrounds
of solitary waves) is strictly negative.  Indeed, about a solitary wave with asymptotic end state $k>0$ and wave speed $c>0$ the group velocity
is
\[
c_g(\ell)=-c+\frac{k(4-\ell^2+\ell^4)}{(1+\ell^2)^2},
\]
which is readily seen to be strictly negative for all wavenumbers $\ell\in\RM$ when $k<c/4$ (a requisite in our upcoming existence theory\footnote{See Lemma \ref{L1} below.}).
As such, small amplitude dispersive waves should travel to the left in a coordinate frame moving with a given solitary wave.
Consequently, it is natural to expect that the dominant solitary wave will ``outrun" the distortions induced by small initial perturbations.  In particular,
if perturbations are measured in norms with exponential weights that vanish at $-\infty$ then smaller solitary waves as well as dispersive radiation
generated by small perturbations should asymptotically vanish as $t\to\infty$.  Indeed, note for instance that if $f\in L^2(\RM)$ and $c>0$ then
\[
\lim_{s\to\infty}\left\|e^{a\xi} f(\xi+st)\right\|_{L^2(\RM;d\xi)}=0
\]
for all ${{a>0}}$.

With the above motivation in mind, the outline for our work is as follows.  Letting $u_0$ be such a smooth soliton solution
with wave speed $c>0$, the linearization of \eqref{DP} about such a solution $u_0$ with wave speed $c>0$ is an evolution equation of the form
\begin{equation}\label{intro_lin}
v_t=\mathcal{A}[u_0]v
\end{equation}
where $\mathcal{A}[u_0]$ is an integro-differential operator with {{coefficients}} depending on the background solitary wave $u_0$.   Motivated
by the above, we consider \eqref{intro_lin} posed on an exponentially weighted space
\[
L^2_\alpha(\RM):=\left\{f : e^{\alpha \xi}f(\xi)\in L^2(\RM;d\xi)\right\}
\]
for appropriate $\alpha>0$ and equipped with its natural weighted norm.  The first primary goal of our analysis is to study the spectrum
of the linearized operator $\mathcal{A}[u_0]$ posed on such weighted spaces $L^2_\alpha(\RM)$.  To this end, we study the essential
and point spectrum separately.  

Regarding the essential spectrum, we note that (as is standard for Hamiltonian systems) when $\alpha=0$, corresponding to the essential
spectrum posed on $L^2(\RM)$, {{it}} is confined completely to the imaginary axis.  For $\alpha>0$ small, the essential spectrum on $L^2_\alpha(\RM)$
is moved into the open left half plane and remains so for all $0<\alpha<\alpha_{\rm crit}$ for some critical $\alpha_{\rm crit}\in(0,1)$ depending on the wave
speed $c>0$ and the asymptotic end state $k>0$ for the underlying solitary wave $u_0$.  
For $\alpha\in(\alpha_{\rm crit},1)$ the essential spectrum will extend into the open right half complex plane, indicating an essential instability.
Finally, although it is not our main concern here, we note that 
when $\alpha>1$ the essential spectrum is again confined to the open left half plane\footnote{As we will see, the essential spectrum
undergoes a singularity when $\alpha=1$.}.  This analysis is carried out in detail in Section \ref{S:ess_spec}.

Continuing, we consider the point spectrum of $\mathcal{A}[u_0]$ posed on $L^2_\alpha(\RM)$ when $0<\alpha<\alpha_{\rm crit}$.  We first note
that the spectral stability result of \cite{Li2019} implies that there are no $L^2_\alpha(\RM)$-eigenvalues of $\mathcal{A}[u_0]$
with strictly positive real part.  As a next main step, we establish the stronger spectral stability result that, in fact, there exists an $\eta>0$
such that
\[
\sigma_{L^2_\alpha}\left(\mathcal{A}[u_0]\right)\cap\left\{\lambda\in\CM:\realpart{\lambda}>-\eta\right\}=\{0\}
\]
and that, furthermore, $\lambda=0$ is an isolated $L^2_\alpha(\RM)$-eigenvalue with algebraic multiplicity two and geometric multiplicity one.
To do this, we first show that $\lambda=0$ is the only possible eigenvalue on the imaginary axis.  Using complete integrability of the DP equation,
we use a new relationship showing that one can construct solutions of the linearized equation \eqref{intro_lin} through an appropriate
quadratic combinations of the eigenfunctions of the Lax pair\footnote{Such a ``squared eigenfunction connection" is common in completely
integrable models.  See, for example, \cite{D1,D2,D3,D4,Calini11,Keith11,Ivey08,Ivey2016,Kapitula2002}.
In the case of the DP equation, however, such a connection seems to be new.}.  Unfortunately, however,
it is not known if all eigenfunctions can be constructed in this way: indeed, while the ``completeness" of the squared eigenfunction connection
is known in many cases, such a completeness result is not available for the DP equation.  Nevertheless, we can, in the genus 1 (i.e. traveling wave) case
obtain a complete set of solutions for the eigenvalue problem \eqref{intro_lin} from the squared eigenfunction connection when the spectral
parameter is restricted to the imaginary axis.  This, in turn, is achieved by studying the asymptotics of {\red{eigenfuntions to}} the Lax pair at spatial infinity.
This analysis is performed in Section \ref{S:complete}.  Finally, in Section \ref{s42} we use analytic properties of the Evans function to
show there is a uniform gap between non-zero possible point spectrum and the imaginary axis.  

With the above spectral analysis in hand, we then turn in Section \ref{S:LinDecay} to studying the linear dynamics of \eqref{intro_lin} posed on $L^2_\alpha(\RM)$.
Through appropriate resolvent estimates, we first establish that $\mathcal{A}[u_0]$ generates a $C_0$-semigroup on $L^2_\alpha(\RM)$.
Using a result of Pr\"uss, we upgrade our spectral analysis on $L^2_\alpha(\RM)$ to establish our main result (which we state somewhat formally here).

\begin{theorem}\label{T:main}
Let $u_0(\cdot;k,c)$ be a smooth solitary wave solution of the DP equation \eqref{DP}, as constructed in Lemma \ref{L1}, and let $0<\alpha<\sqrt{(c-4k)/(c-k)}$ be fixed
and $\eta>0$ be such that
\[
{\rm Re}\left(\sigma_{L^2_\alpha(\RM)}\left(\mathcal{A}[u_0]\right)\setminus\{0\}\right)<-\eta.
\]
Further, let $\Pi:L^2_\alpha(\RM)\to {\rm gKer}\left(\mathcal{A}[u_0]\right)$ be the rank-2 spectral projection onto the generalized kernel
of $\mathcal{A}[u_0]$ on $L^2_\alpha(\RM)$.  Then there exists a constant $C>0$ such that
\[
\left\|e^{\mathcal{A}[u_0]t}\left(1-\Pi\right)f\right\|_{L^2_\alpha(\RM)}\leq C e^{-\eta t}\|f\|_{L^2_\alpha(\RM)}
\]
{{for $f\in L^2_\alpha(\RM)$ and}} for all $t>0$.
\end{theorem}

A precise statement of our main theorem is given in Proposition \ref{P:lin_decay}.  Before continuing, we interpret Theorem \ref{T:main} in terms 
of the expected dynamics of solitary waves.  As will be shown in Section \ref{S:pt_spec}, the generalized kernel of $\mathcal{A}[u_0]$ on $L^2_\alpha(\RM)$
is generated by the continuous symmetries of the underlying PDE: specifically, spatial translation invariance and wave speed variation.  In particular,
using the notation from Theorem \ref{T:main} we have that for each $f\in L^2_\alpha(\RM)$ there exists constants $\gamma,\beta\in\RM$
such that
\[
\Pi e^{\mathcal{A}[u_0]t}f= (\gamma-\beta t)u_0'+\beta\partial_cu_0. 
\]
Consequently, if $u(\cdot,t)$ is a solution of the DP equation \eqref{DP} with  $u(\cdot,0)-u_0(\cdot;k,c)$ sufficiently small in $L^2_\alpha(\RM)$ then it is natural to expect
the solution will evolve for large time as approximately
\begin{align*}
u(x-ct,t)&\approx u_0(x-ct;k,c)+e^{\mathcal{A}[u_0]t}\left(u_0(x-ct;k,c)-u(x-ct,0)\right)\\
&= u_0(x-ct;k,c)+e^{\mathcal{A}[u_0]t}\left(\Pi+(I-\Pi)\right)\left(u_0(x-ct;k,c)-u(x-ct,0)\right)\\
&\approx u_0(x-ct;k,c)+\left(\gamma+\beta t\right)u_0'(x-ct)+\beta\partial_cu_0(x-ct)+\mathcal{O}\left(e^{-\eta t}\right)\\
&\approx u_0\left(x-(c+\beta)t+\gamma;k,c+\beta\right)+\mathcal{O}\left(e^{-\eta t}\right){{,}}
\end{align*}
where here the  approximations are in the $L^2_\alpha(\RM)$-sense\footnote{Note that the first approximation above is just linear approximation,
while the final approximation holds by Taylor series.}.  That is, by Theorem \ref{T:main} it is natural to expect
that small perturbations of a given solitary wave $u_0(\cdot;k,c)$ will converge in $L^2_\alpha(\RM)$ 
to a nearby solitary wave with a slight change in phase and wave speed as $t\to\infty$.

\begin{remark}
As mentioned above, we are at the moment unable to extend the linear asymptotic stability result in
Theorem \ref{T:main} to a result at the nonlinear level.  This will be described in Section \ref{S:nonlinear_try} below.  We believe
the spectral and linear analysis here is an important first step in eventually establishing such
a nonlinear asymptotic stability result.
\end{remark}

\

\noindent
{\bf Acknowledgments:}   The authors would like to thank Ming Chen and
Zhong Wang for helpful conversations regarding their work \cite{CLLW24}.
SL would like to thank Rossen Ivanov for helpful discussions.
The work of MAJ was partially supported by the NSF under grant DMS-2510069.  {{The work of SL was partially supported by the NSF under
grant DMS-2606586 as well as by a Simons Foundation Collaboration Grant for Mathematicians (Award \# 420847)}}.  {The authors are grateful to the referees for their careful reading of the manuscript and for their several
useful observations and suggestions.}

\section{Preliminary Properties of the DP Equation}\label{s1}

In this section, we review some elementary properties of DP equation \eqref{DP}.  We first discuss the existence of smooth solitary wave solutions of the DP equation \eqref{DP},
and see in particular that they necessarily exist on non-zero backgrounds (i.e. they have non-zero asymptotic values at spatial infinity).  
We then introduce the conserved quantities and Hamiltonian structures that will be useful in our work.

\subsection{Existence of Smooth Solitary Waves}

The existence of smooth solitary wave solutions of the DP equation \eqref{DP} has been presented in detail in \cite{LP22smooth}, and relies
on elementary phase plane analysis.  For completeness, we briefly recall the relevant parts of the existence theory.

Traveling wave solutions of \eqref{DP} correspond to solutions of the form
\[
u(x,t) = u_0(x-ct){{,}}
\]
where here $u_0(\cdot)$ is the wave profile profile and $c>0$ is the wave speed.  Introducing the traveling coordinate frame  $\xi=x-ct$, it follows
that the profile $u_0$ is necessarily a solution to the ODE
\[
-c(u_0-u_0'')'+4u_0u_0'=3u_0'u_0''+u_0u_0''',
\]
or, equivalently,
\begin{equation}\label{e:profile1}
-(c-u_0)(u_0-u_0'')'+3u_0'(u_0-u_0'')=0,
\end{equation}
where here prime denotes differentiation with respect to $\xi$.
As described in \cite{LP22smooth}, assuming that $u_0(\xi)<c$ for all $\xi\in\RM$ and multiplying the above by 
the integrating factor $(c-u_0)^2$ allows the above to be rewritten as
\begin{equation}\label{e:a}
(c-u_0)^3(u_0-u_0'')=a
\end{equation}
for some constant of integration $a$.  This, in turn, can be integrated again and rewritten in the 
quadrature form 
\begin{equation}\label{e:quad}
(u_0')^2=E-\left(-u_0^2+\frac{a}{(c-u_0)^2}\right)=:E-V(u_0;a,c).
\end{equation}
where again $E$ is an appropriate constant of integration.
By studying the properties of the effective potential $V$ above, the authors in \cite[Theorem 1.1]{LP22smooth} establish the following result.

\begin{lemma}
	\label{L1}
	For fixed $c>0$ there exists a one-parameter 
	family of smooth solitary waves of \eqref{DP} $u=u_0(\xi)$, $\xi=x-ct$, with profile $u_0 \in C^{\infty}(\mathbb{R})$ satisfying $u_0'(0) = 0$ and 
	$u_0(\xi) \to k$ as $|\xi| \to \infty$ if and only if the arbitrary parameter $k$ belongs 
	to the interval $(0, c/4)$. Moreover,  the solitary wave profile $u_0$ satisfies
	\begin{equation}
	\label{smooth-soliton}
	0 <	u_0(\xi) < c, \qquad  \mu(x)=u_0(\xi)-u_0''(\xi)>0, \qquad \xi \in \mathbb{R},
	\end{equation}
	and the family is smooth with respect to parameters $k$ in $(0, c/4)$ and $c>0$.
\end{lemma}

The above  establishes the existence of a two-parameter family (up to spatial translations) 
of smooth solitary wave solutions $u_0(\cdot;k,c)$ with wave speed $c>0$ and asymptotic end state $k\in(0,c/4)$.  In particular, note that all such solitary
waves necessarily exist on non-zero backgrounds.  Indeed, it is well known that non-trivial smooth solutions of \eqref{DP} that go to zero as $x\to\pm\infty$
do not exist and, instead, one finds peaked solutions (so-called Peakons): see, for example, \cite{LP22smooth} and references therein.  
In our work, we exclusively study the stability of smooth structures,
and hence necessarily work with solitary waves on a non-zero background.  
An implicit formula for that solution is obtained in \cite[Section 5]{Constantin2010}, although we will not make use of this representation here.

{\red{
\begin{remark}\label{re:cp}
Restricting attention to the case $c>0$ entails no loss of generality. Indeed, traveling waves with negative velocity can be obtained from the profile equation \eqref{e:profile1} through the symmetry transformation
\[
c\mapsto -c, \qquad u_0\mapsto -u_0.
\]
At the level of the full DP equation \eqref{DP}, this corresponds to the discrete symmetry
\[
t\mapsto -t, \qquad u\mapsto -u.
\]
Consequently, solutions with $c<0$ are obtained directly from those with $c>0$ and therefore inherit the same stability properties established in this paper. A discussion of this symmetry and of the additional traveling-wave solutions arising in the case $b<1$ can be found in Remarks~7 and~11 of \cite{LP22smooth}. The orbital stability of these other traveling-wave solutions is further investigated in \cite{Li2024stability}.
\end{remark}
}}

We further note that if $u_0(\cdot;k,c)$ is a smooth solitary wave as constructed in Lemma \ref{L1}, then taking $\xi\to\infty$ in \eqref{e:a} and \eqref{e:quad}
yield
\[
a= k(c-k)^3,~~E=kc-2k^2.
\]
In particular, from \eqref{e:quad} we then find that the global maximum of $u_0$ is given explicitly as
\begin{equation}\label{e:max}
\max_{\xi\in\RM}u_0(\xi;k,c) = c-k+\sqrt{ck}.
\end{equation}
which, for a fixed $k\in\RM$, is clearly a strictly increasing function of $c$ for $c>4k$.  It follows that on a fixed background $k>0$ that 
taller solitary waves $u_0(\cdot;k,c)$  move faster than shorter ones\footnote{More precisely, the wave height, defined as $\max_{\xi\in\RM}u_0(\xi;k,c)-k$,
is a strictly increasing function of the wave speed.}.  As mentioned in the introduction, this is a key piece of evidence that the solitary
waves of the DP equation \eqref{DP} are asymptotically stable in exponentially weighted spaces.

\

Before continuing, we make a series of remarks.

\begin{remark}
Associated with the smooth, even solitary waves $u_0$ from Lemma \ref{L1}, we define their momentum densities as $\mu=u_0-u_0''$.  By construction,
such $\mu$ are smooth, even functions satisfying $\mu'(0)=0$, $\mu''(0)<0$, and $\mu(\xi)>0$ for all $\xi\in\RM$.  Further, we note that
\[
\lim_{\xi\to\infty}\mu(\xi)=\lim_{\xi\to\infty}u_0(\xi)=k.
\]
In our work we will make use of both $u_0$ and its momentum density $\mu$.  
\end{remark}

\begin{remark}\label{R:decay}
It is important to observe that the solitary wave $u_0=u_0(\cdot;k,c)$ exhibits exponential decay to its asymptotic end state as $\xi\to\pm\infty$.
Indeed, from \eqref{e:profile1} we see that\,\footnote{This follows from linearizing \eqref{e:profile1} about the constant end state and using the Stable and Unstable Manifold Theorems.  Note the decay rates $r_{\pm}$ are precisely determined by the roots of $P(0,r)=0${{, with $P$ defined in \eqref{Pdef}}}.}
\[
 u_0(\xi;k,c)-k\sim e^{r_\pm \xi}\;\;\text{as}\;\;\xi\to\mp\infty,\;\;\text{where}\;\;r_\pm=\pm\sqrt{\frac{c-4k}{c-k}}.
\]
This observation will be important in our forthcoming work, where we will be considering $u_0$ and its derivatives in exponentially
weighted $L^2$-spaces.
\end{remark}

\begin{remark}
The smooth solitary waves $u_0(\cdot;k,c)$ constructed above can be mapped to smooth solitary wave solutions on a zero background to a (modified) DP equation
via a straightforward Galilean transformation. 
Indeed, a quick calculation shows that if $u_0(\cdot;k,c)$ is a smooth solution of the profile equation \eqref{e:profile1} as constructed above, then
\[
\psi_0(x;k,c) = \frac{2}{3}\left(u_0\left(x;k,c\right)-k\right)
\]
is a smooth solution of
\[
\frac{2}{3}(c-k)\left(\psi-\psi_{xx}\right)_x+3\psi_x\psi_{xx}-4\psi\psi_x+\psi\psi_{xxx}-2k\psi_x=0
\]
satisfying $\lim_{x\to\pm\infty}\psi_0(x)=0$.  In particular, the profiles $\psi_0$ generate smooth solitary traveling wave solutions on a zero background
to the (modified) DP equation
\begin{equation}\label{DP_mod}
\left(U_t-U_{txx}\right)=3U_x U_{xx}-4UU_x+UU_{xxx}-2kU_x
\end{equation}
with wave speed $\frac{2}{3}(c-k)$.  It follows that known results for the stability of (zero background) smooth solitary wave solutions
of \eqref{DP_mod}, including those in \cite{Li2019,Li2024}, immediately translate to stability results of the (non-zero background) smooth solitary waves $u_0(\cdot;k,c)$
of \eqref{DP} constructed above.
\end{remark}

\subsection{Conservation Laws \& Hamiltonian Structures}

The DP equation is well-known to be bi-Hamiltonian and we will make use of both Hamiltonian structures in our work.  The fact that the solitary
waves $u_0(\cdot;k,c)$ exist on non-zero backgrounds, however, requires some care.  To discuss the local dynamics
in \eqref{DP} about $u_0(\cdot;k,c)$ we introduce the set
\[
X_k:=\left\{u-u_{xx}-k\in H^1(\RM):u(x)-u_{xx}(x)>0~{\rm for~all}~x\in\RM\right\}
\]
and note that \eqref{DP} is known to be globally well-posed for $X_k$: see, for example, the discussion in \cite[Section 1]{LP22smooth}.  For solutions
in $X_k$ it well-known that the DP equation admits the conserved quantities
\begin{equation}\label{e:cons1}
\left\{\begin{aligned}
\mathcal{H}(u)&=-\frac{1}{6}\int_\RM\left(u^3-3k^2(u-k)-k^3\right)dx\\
Q(u)&=\frac{1}{2}\int_\RM(u-k)\cdot(1-\partial_x^2)(4-\partial_x^2)^{-1}(u-k)dx
\end{aligned}\right.
\end{equation}
expressed in terms of the solution variable $u$,  as well as the conserved quantities
\begin{equation}\label{e:cons2a}
E(m)=\int_\RM \left(m-k\right)~dx,~~F_1(m)=\int_\RM\left(m^{1/3}-k^{1/3}\right)dx
\end{equation}
and
\begin{equation}\label{e:cons2b}
F_2(m)=\int_\RM\left[\left(\frac{m_x^2}{9m^2}+1\right)m^{-1/3}-k^{-1/3}\right]dx,
\end{equation}
which are expressed in terms of the momentum density variable $m=u-u_{xx}$ and the integrands are normalized to ensure integrability.  One can readily verify that each of the above functionals are well-defined
and smooth on $X_k$.

The conserved quantities $\mathcal{H}$ and $Q$ are related to the Hamiltonian formulation of \eqref{DP} given by
\begin{equation}\label{e:ham1}
u_t=J\frac{\delta\mathcal{H}}{\delta u}
\end{equation}
where
\[
{J=\partial_x\left(1-\partial_x^2\right)^{-1}\left(4-\partial_x^2\right).}
\]
In this formulation, the conserved quantity $Q$ is related to the translation invariance as
\begin{equation}\label{e:translation}
J\frac{\delta Q}{\delta u}=u_x.
\end{equation}
In \cite{Li2019,Li2024}, the authors studied the spectral and orbital stability of smooth solitary wave solutions of the DP equation using the above Hamiltonian structure.  Importantly,
in \cite[Lemma 4.2]{Li2019} the following key monotonicity result was established as part of their spectral stability study:

\begin {lemma} \label{L:non-deg}
For each $c>0$ and $k\in(0,c/4)$ we have
\[
\frac{\partial}{\partial c} Q\left(u_0(\cdot;k,c)\right)>0.
\]
\end{lemma}

The above will be key to our forthcoming spectral analysis, specifically, to our characterization of the generalized kernel
associated to the linearization of \eqref{DP} about a given solitary wave $u_0$.  

\begin{remark}
An important observation is that the above Hamiltonian structure is not sufficient to control the $H^1$-norm of solutions of \eqref{DP}, 
which may be desired in a full nonlinear stability study.
For this reason, we note that the DP equation is in fact bi-Hamiltonian, admitting a completely different Hamiltonian structure from \eqref{e:ham1}
in terms of the momentum density.  Indeed, it  can be readily checked that \eqref{DP} can be rewritten in terms of the momentum density variable $m=u-u_{xx}$ as
\[
m_t+um_x+3mu_x=0,
\]
which can be recognized as the Hamiltonian system
\[
m_t=J_m\frac{\delta E}{\delta m}
\]
where  $E$ is the total (conserved) mass defined in \eqref{e:cons2a} above and
\[
J_m=-\frac{1}{2}\left(3m\partial_x+m_x\right)\left(1-\partial_x^2\right)^{-1}{{\partial_x^{-1}}}\left(3\partial_xm-m_x\right)
\]
is a state-dependent skew-adjoint operator on $L^2(\RM)$.  In this formulation, it can be readily seen that $F_1$ and $F_2$ are 
formal {{Casmirs}} for the associated Hamiltonian flow: see the Appendix in \cite{LP22smooth}.  Clearly, control
of $m$ in even $L^2(\RM)$ gives control of the solution $u$ in $H^2(\RM)$.  
This momentum density based Hamiltonian
formulation has been utilized in several stability works including \cite{Ehrman24,LP22smooth} concerning both periodic and solitary waves.
\end{remark}

\subsection{The Linearization}

In this section, we introduce the appropriate linearized operators.  To this end, note that moving to a co-moving
frame $\xi=x-ct$ and linearizing \eqref{DP} about a given (now stationary) solitary wave solution $u_0(\cdot;k,c)$ leads to the linear evolution equation
\begin{equation}  \label{DPL} 
v_t - v_{t\xi\xi} =  c(v-v'')'+3u_0' v_{\xi\xi}+3u_0''v_\xi - 4u_0v_\xi - 4u_0' v + u_0v_{\xi\xi\xi}+ u_0'''v,
\end{equation}
where $v(\cdot,t)$ belongs to an appropriate class of perturbations for each time $t\geq 0$, and where $'$ denotes differentiating with respect 
to the traveling variable $\xi$.
As \eqref{DPL} is autonomous in time, we seek separated solutions of the form 
\[
v(\xi,t)=v(\xi)e^{\lambda t},~~v(\cdot)\in L^2(\RM),~~\lambda\in\CM
\]
leading to the spectral problem
\begin{equation} \label{DPLeig} 
\lambda(v - v'') = c(v-v'')'+3u_0' v''+3u_0'' v' - 4u_0 v' - 4u_0' v + u_0 v'''+ u_0''' v
\end{equation}
considered again on an appropriate class of perturbations.  A useful observation is that the spectral problem \eqref{DPLeig} can be 
rewritten { (after the trivial transformation $\xi\mapsto -\xi$)} as
\[
\partial_\xi(4-\partial_\xi^2)\left(u_0-c+3c(4-\partial_\xi^2)^{-1}\right)v=\lambda (1-\partial_\xi^2)v
\]
and hence it is equivalent to study the  spectral problem
\begin{equation}\label{e:Adef}
\mathcal{A}[u_0]v=\lambda v,~~~\mathcal{A}[u_0]:=J\mathcal{L}[u_0]
\end{equation}
where here 
\begin{equation}\label{e:JLdef}
J=\partial_\xi(1-\partial_\xi^2)^{-1}(4-\partial_\xi^2),~~\mathcal{L}[u_0]:={{u_0-c+3c(4-\partial_\xi^2)^{-1}}}.
\end{equation}

In this section, our goal is to study the spectrum of $\mathcal{A}[u_0]$ on both $L^2(\RM)$ as well as in the exponentially weighted
space $L^2_{\rm \alpha}(\RM)$.  An important initial observation is that the operators $J$ and $\mathcal{L}[u_0]$ are skew symmetric and symmetric on $L^2(\RM)$, respectively.
As such, on $L^2(\RM)$ the operator $\mathcal{A}[u_0]$ is in Hamiltonian form.  As  a consequence, one can directly verify 
that if $\lambda\in\CM$ is in the $L^2(\RM)$-spectrum of $\mathcal{A}[u_0]$, then so are $-\lambda$ and $\pm\bar{\lambda}$.  That is, the $L^2(\RM)$-spectrum
of $\mathcal{A}[u_0]$ is necessarily symmetric about both the real and imaginary axis.  It follows that $u_0$ is spectrally
stable to perturbations in $L^2(\RM)$ if and only if the $L^2(\RM)$-spectrum is confined to the imaginary axis, i.e. the solitary waves $u_0(\cdot;k,c)$
are at best ``marginally" spectrally stable in $L^2(\RM)$.  With this in mind, the authors in \cite{Li2019} used the Hamiltonian structure \eqref{e:ham1}
along with the monotonicity condition in Lemma \ref{L:non-deg} to establish the following spectral stability result.

\begin{theorem}[\cite{Li2019}]\label{T:spec_stab}
The smooth solitary wave solutions $u_0(\cdot;k,c)$ are spectrally stable to perturbations in $L^2(\RM)$.  In particular,
\[
\sigma_{L^2(\RM)}\left(\mathcal{A}[u_0]\right)= i\RM.
\]
\end{theorem}

\begin{remark}
For completeness, we note that the spectral problem \eqref{e:Adef} is directly related to the Hamiltonian formulation \eqref{e:ham1}.  Indeed, 
recalling the relation \eqref{e:translation} we see that $u_0(\cdot;k,c)$ is a stationary solution of the nonlinear evolution equation
\[
u_t=J\left(\frac{\partial\mathcal{H}}{\partial u}(u)+c\frac{\partial Q}{\partial u}(u)\right).
\]
Linearizing the above about $u_0$ leads to the spectral problem
\[
\lambda v = J\left(\frac{\partial^2\mathcal{H}}{\partial^2 u}(u_0)+c\frac{\partial^2 Q}{\partial^2 u}(u_0)\right)v = J\mathcal{L}[u_0]v,
\]
where the last equality can be verified by a direct calculation.
The spectral properties of $\mathcal{L}[u_0]$ were studied in detail in \cite{Li2019}, forming a crucial component of the later
orbital stability study in \cite{Li2024}.
\end{remark}

\section{Analysis of the Essential Spectrum}\label{S:ess_spec}

In this section we begin our linear stability analysis for a given smooth solitary wave $u_0=u_0(\cdot;k,c)$.    As a first step, we study 
the spectral stability of $u_0$ in both $L^2(\RM)$ and in the exponentially weighted space 
\[
L^2_\alpha(\RM):=\left\{v\in L^2(\RM):v(\xi)e^{\alpha\xi}\in L^2(\RM;d\xi)\right\}
\]
equipped with the natural norm
\[
\|v\|_{L^2_\alpha(\RM)}:=\left\|ve^{\alpha\xi}\right\|_{L^2(\RM;d\xi)}
\]
and considered for appropriate values of $\alpha>0$, to be discussed below.
After introducing the appropriate linearized operator, 
we study its essential and point spectrum in both $L^2(\RM)$ and $L^2_{\alpha}(\RM)$.  We then establish one of our main technical results: that $\lambda=0$
is the only purely imaginary eigenvalue of the linearized operator on $L^2(\RM)$.  Note that while previous works have established
the spectral and orbital stability of the solitary waves $u_0$ (see \cite{Li2019,Li2024,LP22smooth}), these results do not necessarily preclude
the existence of multiple purely imaginary eigenvalues.  We then continue the spectral analysis in $L^2_{\alpha}(\RM)$ by studying the algebraic and geometric
multiplicity of the zero eigenvalue in the weighted space, and by establishing the existence of a spectral gap between the imaginary axis and the non-zero spectrum
in the weighted space.

\subsection{The Essential Spectrum on $L^2(\RM)$}
\label{s41}

We begin by studying the essential spectrum of $\mathcal{A}[u_0]$ acting on $L^2(\RM)$.  	While the essential spectrum
is known to be confined to the imaginary axis thanks to Theorem \ref{T:spec_stab}, we carry out the analysis here to
introduce tools that will be helpful later on.

Using that $u_0(\xi;k,c)\to k$ exponentially fast as $\xi\to\pm\infty$ it follows
from the Weyl Essential Spectrum {\red{Theorem}}\footnote{See, for example, Theorem 2.2.6  and Theorem 3.1.11 in \cite{kapitula2013}.} 
that the $L^2(\RM)$-essential spectrum of $\mathcal{A}[u_0]$  is given precisely by the spectrum associated to the asymptotic spectral problem
\begin{equation} \label{DPLeigess} 
\lambda(v - v'') = \frac{d}{d\xi}\left(c(v-v'') +k(v''-4v)\right),
\end{equation}
where the {{wave $u_0$}} has been replaced by its (constant) asymptotic value $k$ and all derivatives of $u_0$ have been replaced with $0$.  
As \eqref{DPLeigess} is a constant coefficient spectral problem, its spectrum on $L^2(\RM)$ is known to be entirely essential (i.e. there is no point spectrum)
and can be determined by Fourier analysis: for more information, see \cite[Chapter 3]{kapitula2013}.  Indeed, rewriting \eqref{DPLeigess}
as
\[
\lambda(v - v'') =\left((c-k)(v-v'')\right)'{-3kv'},
\]
it follows by Fourier analysis that the $L^2(\RM)$-spectrum of \eqref{DPLeigess} is given precisely by the zero set of the polynomial function
\begin{equation}\label{Pdef}
\mathcal{P}(\lambda,r):= (\lambda+r(k-c))(1-r^2){ +3kr},~~\lambda,r\in\CM
\end{equation}
when the parameter $r$ is restricted to be purely imaginary, i.e. $\lambda\in\CM$ is in the $L^2(\RM)$-spectrum for \eqref{DPLeigess} if and only
if $P(\lambda,i\sigma)=0$ for some $\sigma\in\RM$.
This leads to the (linear) dispersion relation
\begin{equation}\label{firstdisp}
{\mathcal{P}}(\lambda,i\sigma)=0\implies
\lambda=\ri \sigma\left(c-\frac{k(4+\sigma^2)}{1+\sigma^2}\right),\;\;\sigma\in\mathbb{R},
\end{equation}
implying that the $L^2(\RM)$-spectrum for \eqref{DPLeigess}, and hence the essential spectrum of $\mathcal{A}[u_0]$ acting on $L^2(\RM)$,
of the entire imaginary axis.  

\

Before continuing, in addition to determining the location of the essential spectrum, we collect some facts about the  
roots of the polynomial ${\mathcal{P}}$ when $|\lambda|$ is large.  These results  will be useful later when studying the spectrum on the weighted spaces, as well as
in our linear stability study.  
Note that when $|\lambda|$ is very large{{, one}} can naturally expect
that the term $3kr$ in \eqref{Pdef} is asymptotically irrelevant, in the sense that for such large $\lambda$ that the roots of
\eqref{Pdef} should be well-approximated by the roots of the reduced polynomial
\[
\widetilde{\mathcal{P}}(\lambda,r)=(\lambda+r(k-c))(1-r^2),
\]
which are easily computable.  To make this rigorous and to provide appropriate error estimates, we use the the following technical lemma.

\begin{lemma}\cite[Lemma 1.20]{Pego}\label{PegL}
Assume that analytic functions $\widetilde{\mathcal{P}}(r)$ and ${\mathcal{J}}(r)$, depending on a parameter $\lambda$, are given, and that 
${\mathcal{P}}(r) = \widetilde{\mathcal{P}}(r) + {\mathcal{J}}(r)$. Assume that as $|\lambda| \to \infty$, there is a (simple) zero $\tilde{r} = \tilde{r}(\lambda)$ of $\widetilde{\mathcal{P}}$, a positive function $\rho(\lambda) \to 0$ as $|\lambda| \to \infty$ and a constant $\rho_0 > 1$ such that for $|r - \tilde{r}| < \rho$,
\[
\widetilde{\mathcal{P}}'(r) = \widetilde{\mathcal{P}}'(\tilde{r})(1 + o(1)) \quad \text{and} \quad {\mathcal{J}}(r) = {\mathcal{J}}(\tilde{r})(1 + o(1))
\]
as $|\lambda| \to \infty$, and $\rho \geq \rho_0 \left|{\mathcal{J}}(\tilde{r}) /\widetilde{\mathcal{P}}'(\tilde{r})\right|$. Then for $|\lambda|$ sufficiently large, ${\mathcal{P}}$ has exactly one root $r_0 = r_0(\lambda)$ satisfying $|r_0 - \tilde{r}| \leq \rho$.
\end{lemma}

Equipped with Lemma \ref{PegL}, we establish the main result for this section.

\begin{lemma}\label{L0}
The essential spectrum of the operator $\mathcal{A}[u_0]$ acting on $L^2(\mathbb{R})$ 
consists of the entire imaginary axis. 
Furthermore, for each fixed $\lambda\in\CM$ with positive real part, the characteristic polynomial $P(\lambda,r)$ 
has two roots  with positive real part and one root with negative real part, which we denote by $r_1(\lambda)$. 
\end{lemma}

\begin{proof}
We already have proven the part about the essential spectrum.  To characterize the roots of $P(\lambda,r)$ with ${{\realpart{\lambda}}}>0$, we
apply Lemma \ref{PegL} to the polynomial \eqref{Pdef} in a very similar fashion as done in \cite[Section 2(c)]{Pego}. To this end, we define
\[
 \widetilde{\mathcal{P}}(r):= (\lambda+r(k-c))(1-r^2)~~\text{ and }~~{\mathcal{J}}(r):= 3kr{{,}}
\]
and note that the three roots of $\widetilde{\mathcal{P}}(r)$ are given explicitly by
\begin{equation}\label{TR}
\tilde{r}_1=-1,\;\;\tilde{r}_{2}=1,\text{ and }\tilde{r}_{3}=\lambda/(c-k).
\end{equation}
Further, noting that 
{ $\widetilde{\mathcal{P}}'(r) = -2\lambda r+(c-k)(3r^2-1)$} we have that
\[
\widetilde{\mathcal{P}}'(\tilde{r}_1)=2 (c - k + \lambda),\;\;\widetilde{\mathcal{P}}'(\tilde{r}_{2})=2 (c - k - \lambda),\;\;
\widetilde{\mathcal{P}}'(\tilde{r}_{3})=\frac{\left(\lambda^2-(c -k)^2 \right)}{c -k}.
\]
From the expressions above, we directly compute that
\eq{
\widetilde{\mathcal{P}}'(r)&=\widetilde{\mathcal{P}}'(\tilde{r}_1)\left(1+\frac{\left(r +1\right) \left(3 (r-1)(c-k)  -2 \lambda \right)}{2 (c - k + \lambda)}\right),\;\;{\mathcal{J}}(r)={\mathcal{J}}(\tilde{r}_1)\left(1+3k(r+1)\right)\\
\widetilde{\mathcal{P}}'(r)&=\widetilde{\mathcal{P}}'(\tilde{r}_2)\left(1+\frac{\left(r -1\right) \left(3 (r+1)(c-k)  -2 \lambda \right)}{2 (c - k - \lambda)}\right),\;\;{\mathcal{J}}(r)={\mathcal{J}}(\tilde{r}_2)\left(1+3k(r-1)\right)\\
\widetilde{\mathcal{P}}'(r)&=\widetilde{\mathcal{P}}'(\tilde{r}_3)\left(1+\frac{\left(r(c -k) -\lambda \right) \left(3 r (c - k) +\lambda \right)}{\lambda^2-(c -k)^2}\right),\;\;{\mathcal{J}}(r)={\mathcal{J}}(\tilde{r}_3)\left(1+3k(r-\lambda/(c-k))\right).
}{cond151}
For each root $\tilde{r}_i$ we choose
\begin{equation}\label{ro}
\rho_i(\lambda):=  \left|{\mathcal{J}}(\tilde{r}_i) /\widetilde{\mathcal{P}}'(\tilde{r}_i)\right|=O\left(|\lambda|^{-1}\right),\;\;i=1,2,3
\end{equation}
for any $\rho_0>1$.
From \eqref{TR}, \eqref{cond151}, and \eqref{ro}, the assumptions of Lemma \ref{PegL} are clearly verified for the polynomial ${\mathcal{P}}(\lambda,r) $ 
and hence the three roots of ${\mathcal{P}}(r,\lambda)$ satisfy 
\eq{
r_1(\lambda)&=-1+O\left(|\lambda|^{-1}\right),\\
r_2(\lambda)&=\lambda/(c-k)+O\left(|\lambda|^{-1}\right),\\
r_3(\lambda)&=1+O\left(|\lambda|^{-1}\right),
}{3roots}
for $\realpart{\lambda}>0$ and $|\lambda|$ sufficiently large.

From \eqref{3roots}, for large enough $|\lambda|$ with positive real part, we have two roots with positive real parts and one with a negative real part.  
 Since the only location in the complex plane where the real part of the solutions of \eqref{Pdef} can change sign is on the imaginary axis\footnote{This follows from
 the characterization of the essential spectrum.  See, for example, Lemma 3.1.10 or Theorem 3.1.11 in \cite{kapitula2013}.}, 
 we have the assertion of the lemma concerning the sign of the real parts of the roots $r=r(\lambda)$ of $P(\lambda,r)=0$ when $\left(\realpart{\lambda}>0\right)$.
\end{proof}

Before continuing, we note for $\lambda$ in the $L^2(\RM)$-essential spectrum of $\mathcal{A}[u_0]$ that the roots polynomial equation
$\mathcal{P}(\lambda,r)=0$ satisfy
\[
\realpart{r_1(\lambda)}<0=\realpart{r_2(\lambda)}<\realpart{r_3(\lambda)}
\]
while, for $\lambda$ to the right of the essential spectrum we have
\[
\realpart{r_1(\lambda)}<0<\realpart{r_j(\lambda)},~~j=2,3.
\]
In particular, we observe that for each $\realpart{\lambda}\geq 0$ the polynomial $\mathcal{P}(\lambda,r)$ has a \emph{unique}
root with negative real part.  This observation will be important in Section \ref{S:pt_spec} below.

\subsection{The Essential Spectrum on $L^2_{\alpha}(\RM)$}
\label{s42}

We now consider the essential spectrum for the operator $\mathcal{A}[u_0]$ when
acting on the exponentially weighted space $L^2_\alpha(\RM)$.  In particular, we show that this essential spectrum is, for appropriate {{choice}} of $\alpha$,
contained in the open left-half plane $\realpart{\lambda}<0$.

As a first step, we make a change of variables to better understand the operator $\mathcal{A}[u_0]$ when acting on $L^2_{\alpha}(\RM)$.
To this end, note by definition that for a given $v\in L^2_\alpha(\RM)$ we can set $w(\xi)=v(\xi)e^{-\alpha \xi}\in L^2(\RM)$.  Making this substitution
in the spectral problem \eqref{DPLeig} leads to the spectral problem
\begin{equation}\label{spec_w}
\lambda w=e^{\alpha\xi}\mathcal{A}[u_0]e^{-\alpha\xi}w=:\mathcal{A}_\alpha [u_0]w
\end{equation}
considered on $L^2(\RM)$.  Note that, in effect, the operator $\mathcal{A}_\alpha[u_0]$ is obtained from $\mathcal{A}[u_0]$ by making
the transformation 
\[
d/d\xi\to d/d\xi-\alpha.
\]
As a result of the above transformation, we see that the spectrum of $\mathcal{A}[u_0]$ on the weighted
space $L^2_\alpha$ is equivalent to the spectrum of the conjugated operator $\mathcal{A}_{\alpha}[u_0]$
acting on $L^2(\RM)$, i.e. we have
\[
\sigma_{L^2_{\alpha}(\RM)}\left(\mathcal{A}[u_0]\right) = \sigma_{L^2(\RM)}\left(\mathcal{A}_\alpha[u_0]\right).
\]
This equivalence will be used throughout our work.

\begin{remark}\label{R:wspec_nosymmetry}
It is straightforward to see that the Hamiltonian structure of \eqref{DP} implies that the conjugated linearized operator $\mathcal{A}_\alpha[u_0]$
can be decomposed as $\mathcal{A}_\alpha[u_0]=J_\alpha\mathcal{L}_\alpha[u_0]$, where\footnote{Compare to \eqref{e:Adef} and \eqref{e:JLdef} in the unweighted case.
Also, even though we restrict to $\alpha>0$ throughout our stability analysis, we take these definitions of $J_\alpha$ and $\mathcal{L}_\alpha[u_0]$ for all $\alpha\in\RM$.}
\begin{align*}
\mathcal{L}_\alpha[u_0]&=(4-(\partial_\xi-\alpha)^2)^{-1}\left((4-(\partial_\xi-\alpha)^2)(u_0-c)+3c\right),\\
J_\alpha&=(\partial_\xi-\alpha)\left(1-(\partial_\xi-\alpha)^2\right)^{-1}\left(4-(\partial_\xi-\alpha)^2\right).
\end{align*}
One can readily check that the adjoint operator satisfies $\mathcal{A}^\dag_\alpha[u_0]=-\mathcal{L}_{-\alpha}[u_0]J_{-\alpha}$ so that, in particular, for $\alpha>0$ the operator
$\mathcal{A}_\alpha[u_0]$ does not have a Hamiltonian structure.  Consequently, the spectrum of $\mathcal{A}_\alpha[u_0]$ is no longer symmetric
with respect to reflections about the imaginary axis (although it is still symmetric about the real axis).
\end{remark}

To proceed, note that the spectral problem \eqref{spec_w} can  be written explicitly as
 \begin{equation}\label{DPLeigessW}
 \lambda((1-\alpha^2)w +2\alpha w'- w'') = \left(\frac{d}{d\xi}-\alpha\right)\left\{c((1-\alpha^2)w +2\alpha w'- w'') +k((4-\alpha^2)w -2\alpha w'- w'')\right\}.  
 \end{equation}
Following the procedure from the previous section, we find that the essential spectrum of the operator $\mathcal{A}_\alpha[u_0]$
on $L^2(\RM)$ is determined by the roots curve in the complex plane defined by the roots $\lambda$ of the polynomial equation
\[
{\mathcal{P}}(\lambda,i\sigma-\alpha)=0
\]
for $\sigma\in\RM$.  By replacing $i\sigma$ with $i\sigma-\alpha$ in the (unweighted) dispersion relation \eqref{firstdisp} it follows
that the essential spectrum of $\mathcal{A}_a[u_0]$ is given by
\begin{equation}\label{lmroots}
\begin{aligned}
\lambda&=(\ri \sigma-\alpha)\left(c-\frac{k(4-(i\sigma-\alpha)^2)}{1-(i\sigma-\alpha)^2}\right)\\
&=(\ri \sigma-\alpha)\left(c-k-\frac{3k(1+\sigma^2-\alpha^2 - 2\ri\sigma\alpha)}{(1+\sigma^2-\alpha^2)^2+4\sigma^2\alpha^2}\right),~~\sigma\in\RM.
\end{aligned}
\end{equation}
In particular, the real and imaginary parts of $\lambda=\lambda(\sigma)$ on the essential spectrum are given by
\begin{subequations}\label{RI}
  \begin{align}
\realpart{\lambda}&=-\alpha\left(c-k+3k\left(\frac{\alpha^2+\sigma^2-1}{(1+\sigma^2-\alpha^2)^2+4\sigma^2\alpha^2}\right)\right)\label{Rel}\\
\impart{\lambda}&=\sigma\left(c-k-3k\left(\frac{\alpha^2+\sigma^2+1}{(1+\sigma^2-\alpha^2)^2+4\sigma^2\alpha^2}\right)\right)\label{Ima}.
  \end{align}
\end{subequations}
 Note that for large $\sigma$, the expression for $\realpart{\lambda}$ becomes
\eq{
\realpart{\lambda}\to -\alpha\left(c-k\right)\text{ as }\sigma\to \infty{{,}}
}{asb}
and hence the essential spectrum has a vertical asymptote at $\realpart{\lambda}=-\alpha(c-k)$.  Noting that the existence 
theory from Lemma \ref{L1} implies that $c>k$, it follows that $\alpha>0$ is a necessary condition for the essential
spectrum of $\mathcal{A}_\alpha[u_0]$ to be contained in the open left-half plane.
We further note that the imaginary part of the essential spectrum has a singularity at $\sigma=0$ when $\alpha=1$, and as such we restrict our attention
throughout to the case\footnote{For completeness, we will occasionally make remarks concerning the case when $\alpha>1$.} when $0<\alpha<1$.

Continuing, if $0<\alpha<1$ then the condition ensuring that $\realpart{\lambda}<0$ for all $\sigma\in\mathbb{R}$ is given by
\begin{equation}\label{polycond}
 \left(c -k \right) \sigma^4+\left(2 \alpha^2  (c - k)  +2 c +5 k \right) \sigma^2  +\left(\alpha^2  -1\right) \left(\alpha^2  (c -k)   -c +4 k \right)> 0\text{ for all }\sigma\in\mathbb{R}.
\end{equation}
Treating the above as a quadratic polynomial in $\sigma^2$, and noting that the coefficients of $\sigma^4$ and $\sigma^2$ are both positive\footnote{The $\sigma^2$
coefficient is readily seen to be positive for $\alpha\in(0,1)$ after rewriting it as $2c(\alpha^2+1)+k(5-2\alpha^2)$ and recalling from Lemma \ref{L1} that $k>0$.},
it follows that \eqref{polycond} holds provided that the last term is positive as well, i.e. provided that
\begin{equation}\label{conddef1}
0< \alpha^2< \frac{c-4k}{c-k}.
\end{equation}
It follows for $\alpha$ satisfying \eqref{conddef1} that the essential spectrum of the weighted operator $\mathcal{A}_\alpha[u_0]$ acting on $L^2(\RM)$ is contained
in the open left-half complex plane {{$\realpart{\lambda}<0$}}.  

Concerning the size of the  ``spectral gap'', i.e. the distance between the imaginary axis and the 
(necessarily stable) essential spectrum, we note by elementary calculus that the real part \eqref{Rel} is maximized  as a function of $\sigma$
when either $\sigma\to\pm\infty$ or when
\[
\frac{\partial}{\partial\sigma}{{\realpart{\lambda(\sigma)}}}=2\sigma\frac{4(1-\alpha^2)-(\sigma^2 + \alpha^2 - 1)^2}{((1+\sigma^2-\alpha^2)^2+4\sigma^2\alpha^2)^2}=0.
\]
For $\alpha\in(0,1)$ the above derivative vanishes at the three roots $\sigma=0$ and $\sigma^2=1-\alpha^2  + 2\sqrt{1-\alpha^2}$.  Continuing to
assume \eqref{conddef1} holds and evaluating \eqref{Rel} at each
critical point, we find that $\realpart{\lambda}$ is largest at $\sigma=0$ and that
\begin{equation}\label{e:ess_spec_compare}
\left. \realpart{\lambda(\sigma)} \right|_{\substack{\sigma = 0}}
=-\alpha\left(c-k-\frac{3k}{1-\alpha^2}\right)>-\alpha(c-k)=\lim_{\sigma\to\pm\infty}\lambda(\sigma).
\end{equation}
Consequently, assuming that \eqref{conddef1} holds, the point at $\sigma=0$ is the closest to the imaginary axis.
By a similar analysis as the proof of Lemma \ref{L0}, the above work establishes the following result.

\begin{proposition}\label{L3}
For $\alpha\in(0,1)$, the essential spectrum of the operator $\mathcal{A}_a[u_0]$ acting on $L^2(\RM)$ lies in the open left-half complex plane $\realpart{\lambda}<0$
provided that the weight $\alpha$ satisfies \eqref{conddef1}, in which  case the spectral gap for the essential spectrum is given by 
\begin{equation}\label{SG2}
\Delta_\alpha=\alpha\left(c-k-\frac{3k}{1-\alpha^2}\right)>0.
\end{equation}
\end{proposition}

Note for each $\alpha>0$ satisfying \eqref{conddef1} that the complement of the essential spectrum of $\mathcal{A}_\alpha[u_0]$ 
consists of two disjoint, open, and connected components in $\CM$.  For each such $\alpha$, we let
$\Omega_+(\alpha)$ denote the open such component that includes the closed right half plane $\realpart{\lambda}\geq 0$.
Observe in particular that any $\lambda\in\Omega_+(\alpha)$ either belongs to the resolvent set of $\mathcal{A}_\alpha[u_0]$
or else is an isolated eigenvalue with finite multiplicity.  Further, we have the following observation 
that will end up being helpful later.

{\red{\begin{remark}
In view of Remark \ref{re:cp}, solutions with negative velocity can be obtained by choosing the velocity to be $-c$ and the asymptotic state to be $-k$. Making the substitutions $c\mapsto -c$ and $k\mapsto -k$ at the outset of the analysis leading to \eqref{lmroots}, and then following the same arguments throughout, one finds that, for the traveling wave solution with velocity $-c<0$ and asymptotic state $-k<0$, in Proposition \ref{L3} the parameter $\alpha$ must be chosen in the interval $(-1,0)$. Moreover, the corresponding gap formula remains unchanged except for an overall minus sign.
\end{remark}
}}

\begin{corollary}\label{C:roots}
For each $\alpha>0$ satisfying \eqref{conddef1} the roots of the polynomial $\mathcal{P}(\lambda,r-\alpha)=0$ satisfy
\[
\left\{\begin{aligned}
&\realpart{r_1(\lambda)}<-\alpha=\realpart{r_2(\lambda)}<\realpart{r_3(\lambda)},~~\lambda\in\partial\Omega_+(\alpha)\\
&\realpart{r_1(\lambda)}<-\alpha<\realpart{r_j(\lambda)},~~j=2,3,~~\lambda\in\Omega_+(\alpha){{.}}
\end{aligned}\right.
\]
In particular, for all $0<a<\alpha<\sqrt{(c-4k)/(c-k)}$ we have the containment property
\[
\left\{\lambda\in\CM:\realpart{\lambda}\geq 0\right\}\subset \overline{\Omega_+(a)}\subset\Omega_+(\alpha).
\]
\end{corollary}

\begin{remark}
The fact that \eqref{asb} implies the essential spectrum of $\mathcal{A}_\alpha[u_0]$ is asymptotically vertical presents a 
challenging analytical obstacle  to be overcome if one were to extend, using our methodology, our results to the nonlinear level.  
As we will describe in Section \ref{S:nonlinear_try} below, the presence of the nonlinear 
dispersive term $uu_{xxx}$ in the DP equation \eqref{DP} leads to a loss of derivatives in a contraction-mapping based iteration scheme.  
Unfortunately, these derivatives cannot be regained at a linear level due to \eqref{asb} implying the lack of a linear
smoothing estimate.  See Section \ref{S:nonlinear_try} below for more details.
\end{remark}

\begin{remark}\label{R:decay2}
It is important to note that the the upper bound for $\alpha>0$ in \eqref{conddef1} agrees precisely with the exponential
decay rate of the underlying profile $u_0(\cdot;k,c)$ to its asymptotic end state: see Remark \ref{R:decay} above.
In particular, we note that  the function $u_0(\cdot;k,c)-k$, as well as any derivative in $\xi$ or $c$ of it, will necessarily
belong to $L^2_\alpha(\RM)$ for all $\alpha>0$ satisfying \eqref{conddef1}.
\end{remark}

For completeness, we note the following result concerning the essential spectrum of the operator $\mathcal{A}_\alpha[u_0]$ when $\alpha>1$.

\begin{proposition}\label{P:ess_bigalpha}
When $\alpha>1$, the essential spectrum of the operator $\mathcal{A}_\alpha[u_0]$ acting on $L^2(\RM)$ lies in the open half
plane ${{\realpart{\lambda}}}<0$.  Further, for $\alpha>1$ the spectral gap for the essential spectrum is given by 
\[
\Delta_\alpha=\alpha(c-k),~~\alpha>1.
\]
\end{proposition}
\begin{proof}
Continuing with the notation from the calculations above Proposition \ref{L3}, when $\alpha>1$  we see from \eqref{Rel} that
\[
\realpart{\lambda(\sigma)} <-\alpha(c-k)=\lim_{\sigma\to\pm\infty}\lambda(\sigma)<0
\]
for all $\sigma\in\RM$, proving the entirety of the essential spectrum of $\mathcal{A}_\alpha[u_0]$ indeed lies in the open
left half plane.  Regarding the size of the gap, we note when $\alpha>1$
that \eqref{e:ess_spec_compare} becomes
\begin{equation}
\left. \realpart{\lambda(\sigma)} \right|_{\substack{\sigma = 0}}
=-\alpha\left(c-k-\frac{3k}{1-\alpha^2}\right)<-\alpha(c-k)=\lim_{\sigma\to\pm\infty}\lambda(\sigma),
\end{equation}
completing the proof.
\end{proof}

 \begin{figure}[tbp]
\begin{center}
\includegraphics[scale=0.28]{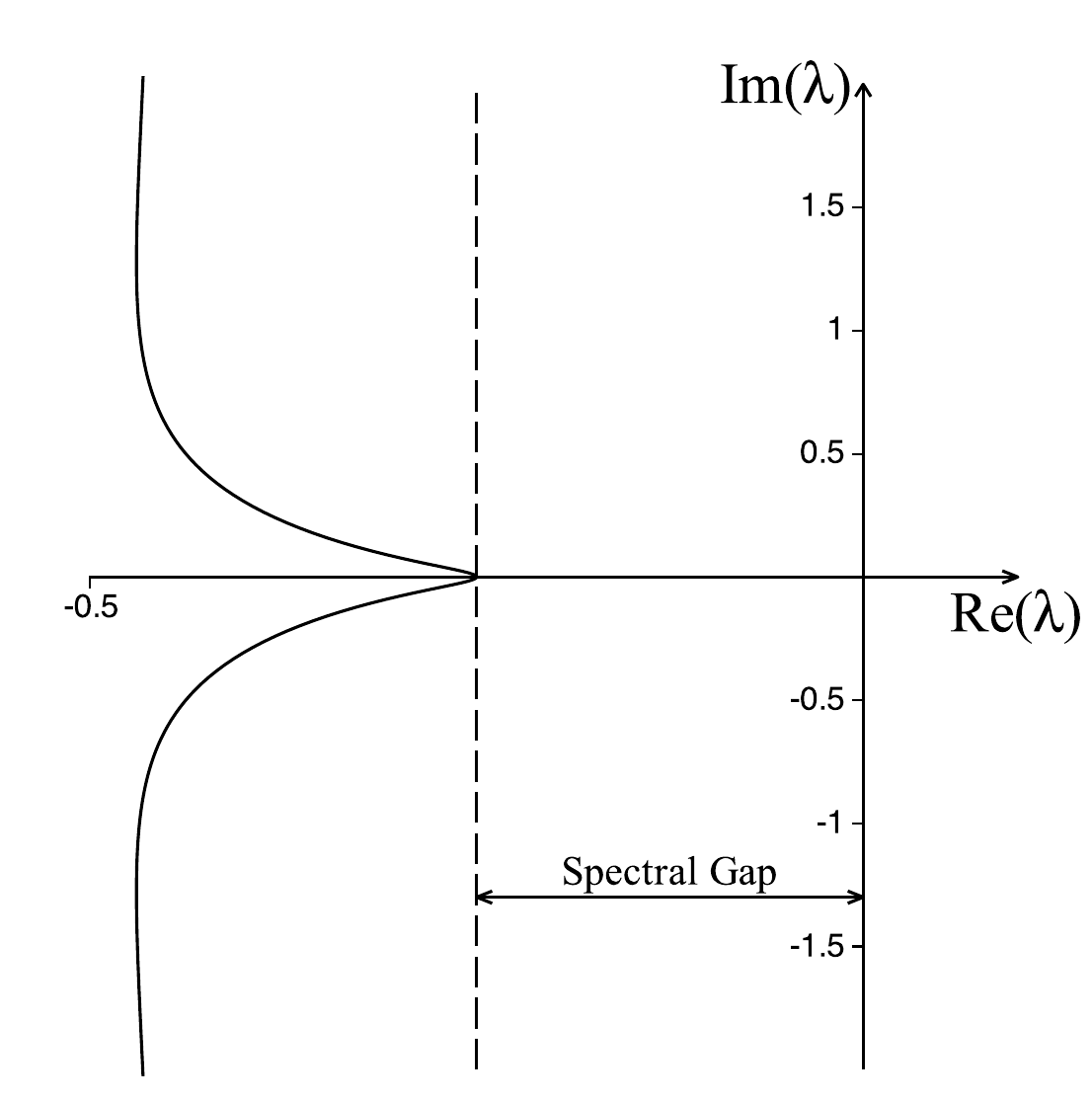}~~~ \includegraphics[scale=0.27]{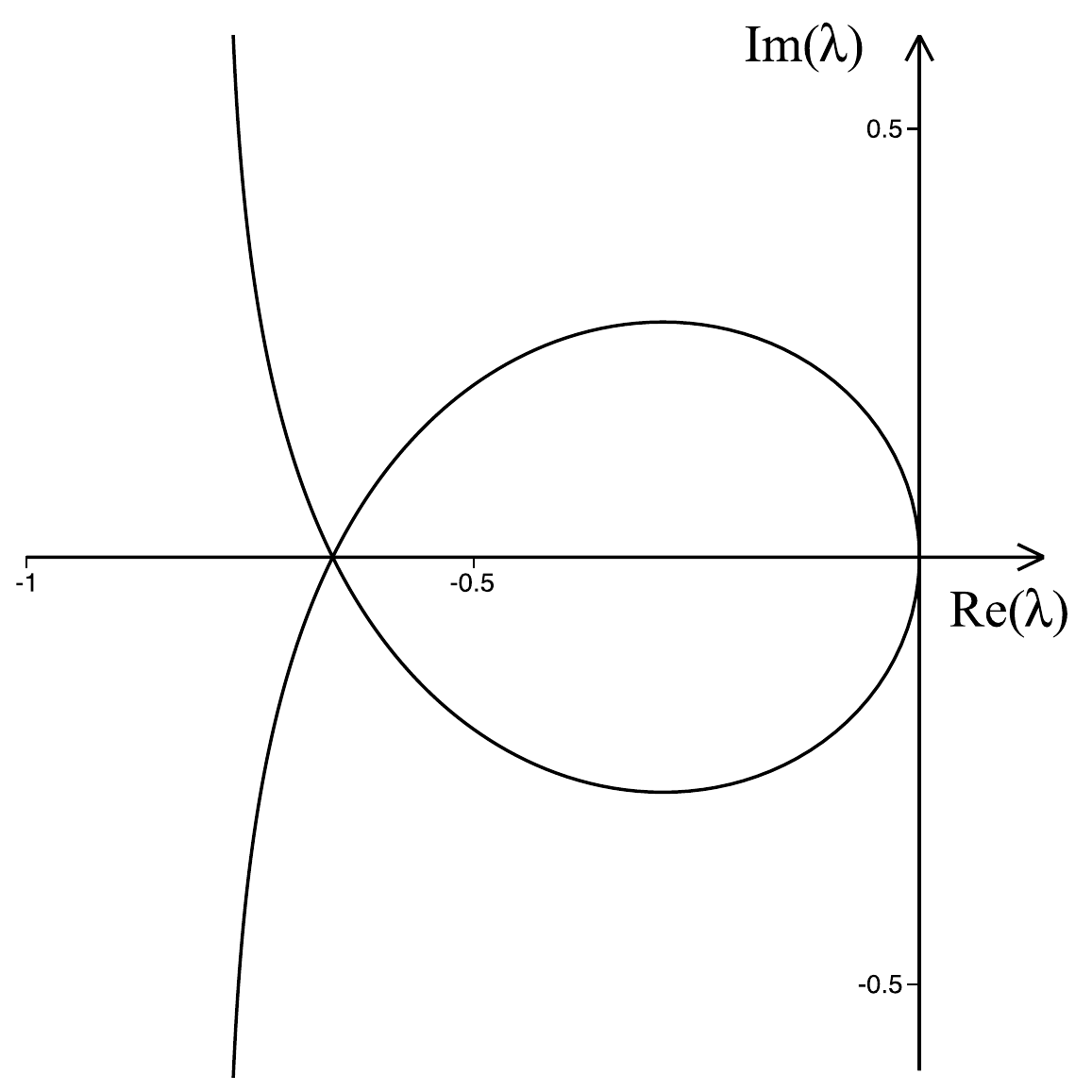}~~~
\includegraphics[scale=0.27]{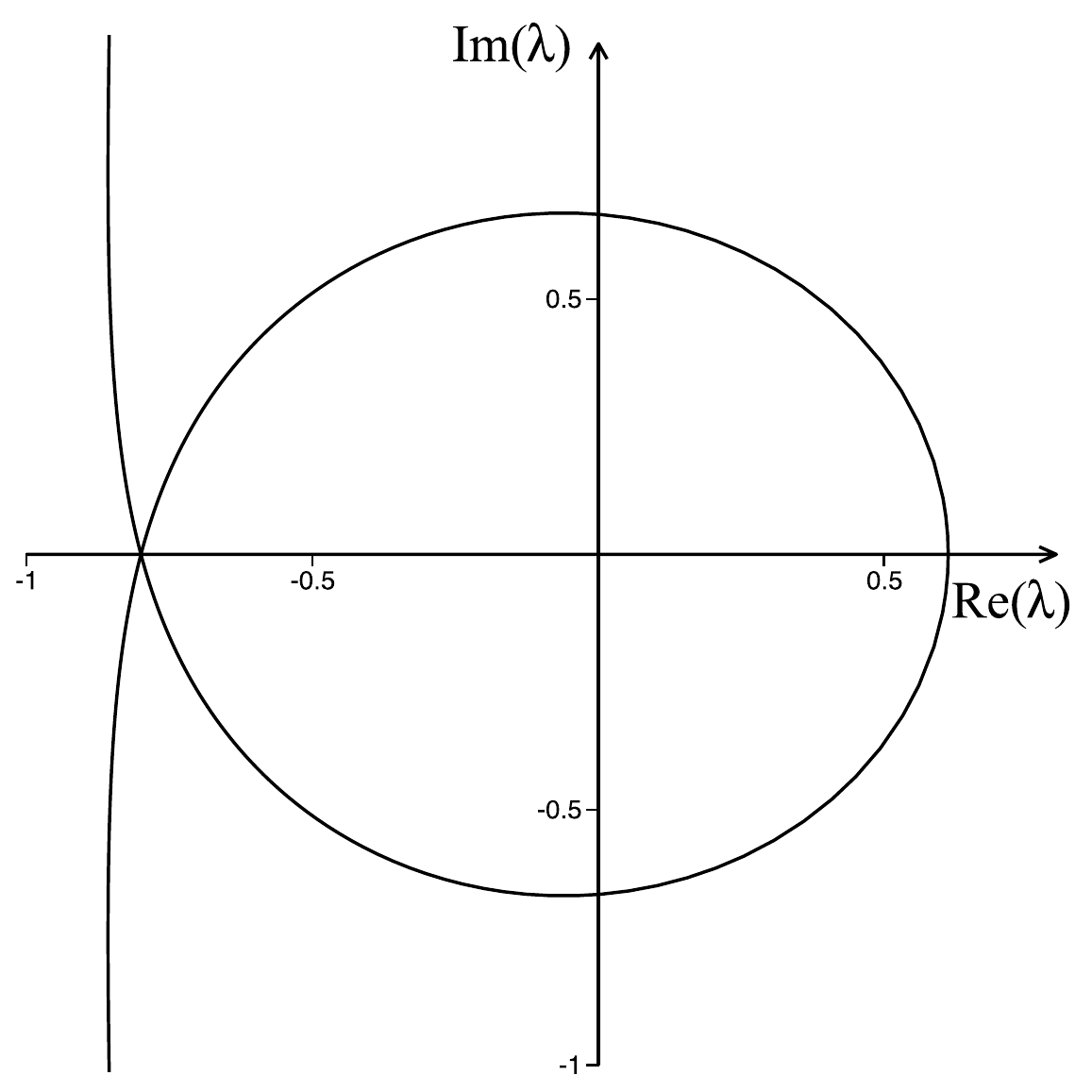}
\end{center}
\caption{ {Graph of the essential spectrum of $\mathcal{A}_\alpha[u_0]$, parametrized by $\sigma$ as given in \eqref{RI} for $0<\alpha<1$.
For these graphs, the underlying wave corresponds to $k=0.1$ and $c=1$, in which case \eqref{conddef1} reduces to $0<\alpha^2<2/3$.  
The  plot on the left takes $\alpha=0.5{<\sqrt{2/3}}$, thus satisfying \eqref{conddef1}, indicating stability of the essential spectrum.  The middle plot  takes $\alpha=\sqrt{2/3}$, thus showing an essential spectrum touching the imaginary axis corresponding to marginal stability.
The plot on the right takes $\alpha=0.9>\sqrt{2/3}$, thus displaying an unstable essential spectrum.
  \label{fig1}}}
\end{figure}

See Figure \ref{fig1} for a numerical calculation of the essential spectrum about the wave $u_0(\cdot;k,c)$ with $(k,c)=(0.1,1)$ in the cases
when $\alpha=0.5$ (when \eqref{conddef1} holds){{, when $\alpha=\sqrt{2/3}$ (case of marginal stability when $\alpha$ is at the upper boundary
of the condition \eqref{conddef1}, and when $\alpha=0.9$ (unstable essential spectrum when $\alpha<1$.  
Figure \ref{fig2} demonstrates an analogous calculation in the case when $\alpha$ is beyond the upper limit in \eqref{conddef1}.}}

 \begin{figure}[tbp]
\begin{center}
\includegraphics[scale=0.35]{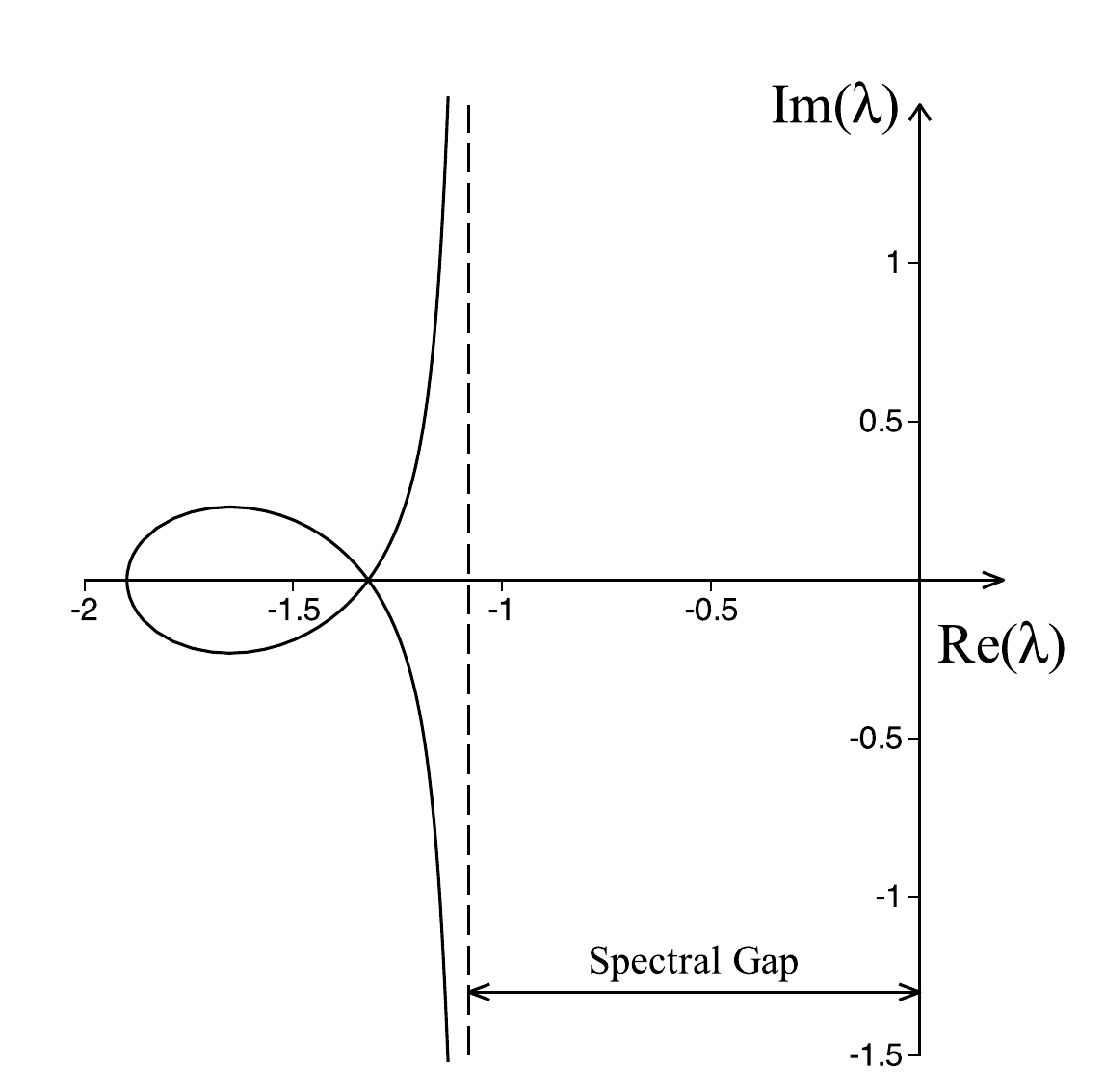}
\end{center}
\caption{ {
Using the same $k$ and $c$ values from Figure \ref{fig1}, this is a graph of the essential spectrum of $\mathcal{A}_\alpha[u_0]$ in the case where $\alpha=1.2$, indicating
the expected stability of the essential spectrum.\label{fig2}}}
\end{figure}

\section{Analysis of the Point Spectrum}\label{S:pt_spec}
\label{s2}

We continue our study of the spectral properties of $\mathcal{A}[u_0]$ by studying its point spectrum on $L^2(\RM)$
and $L^2_\alpha(\RM)$.  
Recall from \cite{Li2019} that all smooth solitary wave solutions $u_0$ constructed in Lemma \ref{L1} are spectrally
stable, i.e. that $\sigma\left(\mathcal{A}[u_0]\right)= i\RM$ (see Theorem \ref{T:spec_stab} above).  
In this section we first refine this result to show
that, in fact, $\lambda=0$ is the only eigenvalue of $\mathcal{A}[u_0]$ acting on $L^2(\RM)$, thus ruling out the possibility
of $\mathcal{A}[u_0]$ having any non-zero eigenvalues embedded in the essential spectrum.  
{{Our analysis will rely}} heavily on the complete integrability of the governing DP equation,
which we will review below.  We then consider the point spectrum in the weighted space $L^2_\alpha(\RM)$,
proving the existence of a spectral gap between the non-zero spectrum and the imaginary axis for appropriate values of $\alpha$.
Finally, we characterize the generalized kernel of the linearized operator $\mathcal{A}[u_0]$ acting on
the weighted space $L^2_\alpha(\RM)$.

\subsection{Weighted vs. Unweighted Point Spectrum}

To begin this section, we start with some important observations regarding the point spectrum of $\mathcal{A}[u_0]$
acting on $L^2(\RM)$ vs acting on $L^2_\alpha(\RM)$.  First, recall as in the previous section\footnote{Specifically,
see \eqref{spec_w} and the surrounding exposition.} that studying the 
spectrum of $\mathcal{A}[u_0]$ on the weighted space $L^2_\alpha(\RM)$ is equivalent to studying the spectrum
of the conjugated operator $\mathcal{A}_\alpha[u_0]$ acting on $L^2(\RM)$.  
The main result of this section is the following equivalence of the 
point spectrum of $\mathcal{A}[u_0]$ and $\mathcal{A}_\alpha[u_0]$ in the closed right half-plane.

\begin{proposition}\label{P:spec_equiv}
Let $u_0(\cdot;k,c)$ be a smooth solitary wave of the DP equation \eqref{DP} as constructed in Lemma \ref{L1},
and let $\alpha>0$ satisfy \eqref{conddef1}.
When $\realpart{\lambda}\geq 0$, we have that $\lambda$ is an $L^2(\RM)$ eigenvalue of $\mathcal{A}[u_0]$
if and only if $\lambda$ is an $L^2(\RM)$ eigenvalue of $\mathcal{A}_\alpha[u_0]$.
\end{proposition}

The above is in fact a direct corollary of Lemma \ref{DA} below.  
Recalling Theorem \ref{T:spec_stab}, it immediately follows that the weighted operator $\mathcal{A}_\alpha[u_0]$
has no point spectrum with positive real part.  For our asymptotic stability analysis, however, we need to additionally
rule out the existence of non-zero purely imaginary eigenvalues of $\mathcal{A}_\alpha[u_0]$, as well as eigenvalues
of $\mathcal{A}_\alpha[u_0]$ with negative real part and arbitrariliy close to the imaginary axis (i.e. we need the existence 
of a spectral gap).  By Proposition \ref{P:spec_equiv}, we can rule out the existence of non-zero imaginary eigenvalues of $\mathcal{A}_\alpha[u_0]$
by showing that $\lambda=0$ is the only $L^2(\RM)$ eigenvalue of $\mathcal{A}[u_0]$ on the 
imaginary axis\footnote{Note that, due to Theorem \ref{T:spec_stab} and the fact that if
$\lambda$ is an $L^2(\RM)$ eigenvalue of $\mathcal{A}[u_]$ then so are $\pm\lambda$ and $\pm\overline{\lambda}$, this equivalently
shows that $\lambda=0$ is the only eigenvalue of $\mathcal{A}[u_0]$ acting on $L^2(\RM)$.  This strong form of spectral
stability is new for the DP equation.}.  This will be shown in the next section by relying on the complete
integrability of the DP equation.  The existence of a spectral gap will then by established through the use of resolvent estimates
 in Section \ref{S:LinDecay} below.

\begin{remark}\label{R:equiv}
When $\realpart{\lambda}<0$ and $\lambda\in\Omega_+(\alpha)$, the above equivalence is no longer true: in particular,
if such a $\lambda$ is an $L^2(\RM)$ eigenvalue of $\mathcal{A}_\alpha[u_0]$ then $\lambda$ need not be an
$L^2(\RM)$ eigenvalue of $\mathcal{A}[u_0]$.  This is a reflection of the fact that spectrum
of $\mathcal{A}_\alpha[u_0]$ is not symmetric about the imaginary axis when $\alpha>0$.  See Remark \ref{R:wspec_nosymmetry} for details.
Consequently, the spectral stability result of Theorem \ref{T:spec_stab} cannot be used to infer the existence of a spectral gap
for $\alpha>0$  satisfying \eqref{conddef1}.
\end{remark}

The proof of Proposition \ref{P:spec_equiv} is based on a study of the Evans function associated with the spectral
problems for $\mathcal{A}[u_0]$ and $\mathcal{A}_\alpha[u_0]$ posed on $L^2(\RM)$.  
To begin,
first note that the system (\ref{DPLeig}) can be written as a first-order linear dynamical system
of the form
\begin{equation}
X'(\xi)=A(\xi,\lambda)\,X(\xi),
\label{linear}
\end{equation}
where $A$ is given explicitly by the $3\times 3$  matrix 
\eq{
A(\xi,\lambda)=
\left(
\begin{array}{ccccc}
0&&1&&0\\
0&&0&&1\\
\left(u_0'''-\lambda-u_0'\right)/(c-u_0)&&\left(3u_0''-4u_0+c\right)/(c-u_0)&&\left(\lambda+3u_0'\right)/(c-u_0)
\end{array}
\right).
}{A}
The asymptotic behavior as $\xi\rightarrow\infty$ of the solutions to (\ref{linear}) 
is determined by the matrix
$A^{\infty}(\lambda)=\lim_{\xi\rightarrow \infty}A(\xi,\lambda)$ which, by Lemma \ref{L1}, is given explicitly by
\[
A^{\infty}(\lambda) = \begin{pmatrix}
0&1&0\\
0&0&1\\
-\lambda/(c-k)&\left(c-4k\right)/(c-k)&\lambda/(c-k)
\end{pmatrix}.
\]
The characteristic polynomial of $A^{\infty}$ is given by 
\eq{{\mathcal{P}}(\lambda,r)= 0,}{Pdef2}
with ${\mathcal{P}}$ given in \eqref{Pdef}.  In particular, noting that
\[
\det\left(A^\infty(\lambda)\right)=-\frac{\lambda}{c-k}
\]
we see for all $\realpart{\lambda}>0$ that the asymptotic
matrix $A^\infty(\lambda)$ has one eigenvalue with negative real part and two eigenvalues with positive real part.  Indeed,
the real parts of the eigenvalues 
of $A^\infty(\lambda)$ can only change when $\lambda$ crosses the imaginary axis (i.e.~the essential spectrum),
and the claim is easy to verify for $\lambda$, say, real and positive\footnote{For such $\lambda$, the determinant calculation above implies that 
the product of all three eigenvalues of $A^\infty(\lambda)$ must be negative.  Since for such $\lambda$ the entries of $A^\infty(\lambda)$ are real, the result trivially follows.}.
Consequently, for $\realpart{\lambda}>0$ the linear system \eqref{linear} has three linearly independent solutions, one decaying exponentially as $\xi\to\infty$
and two decaying exponentially as $\xi\to-\infty$.

Following the notation from Section \ref{s41}, for $\realpart{\lambda} > 0$ we denote the 
unique eigenvalue of $A^{\infty}(\lambda)$ with negative real part as $r_1(\lambda)$ and its corresponding eigenvector by
$v_+$, respectively.
System (\ref{linear}) then has a unique solution $X_+(\xi)$ satisfying\,\footnote{See, for instance,  \cite{codd} for details.}
\eq{
\lim_{\xi\rightarrow \infty} X_{+}(\xi;\lambda)e^{-r_1(\lambda)\,\xi}=v_+.
}{XP}
In this situation where the linear system (\ref{linear}) has only one linearly independent solution decaying as $\xi\to\infty$, 
the definition of the  Evans function requires that we additionally consider the adjoint system \cite{Pego} 
\begin{equation}
Y'(\xi)=-A^T(\xi,\lambda)\,Y(\xi).
\label{adjoint}
\end{equation}
The asymptotic behavior of the solutions of \eqref{adjoint} as $\xi\rightarrow-\infty$ is determined by
the  eigenvalues of the matrix
\[
-\lim_{\xi\to-\infty}A^T(\xi,\lambda)
=-\left(A^\infty(\lambda)\right)^T,
\]
which, by the arguments above, for $\realpart{\lambda}>0$ has exactly one eigenvalue with positive real part and two eigenvalues with negative real parts.
For $\realpart{\lambda} > 0$, the unique eigenvalue of $-\left(A^{\infty}\right)^T$ with positive real part is given by
$-r_1(\lambda)$ and we denote the corresponding eigenvector by
$v_- $, chosen to be normalized so that $v_-\cdot v_+=1$.
System (\ref{adjoint}) then has a unique solution $Y_-$ satisfying
\eq{
\lim_{\xi\rightarrow -\infty} Y_{-}(\xi;\lambda)e^{-r_1(\lambda)\,\xi}=v_-.
}{YM}
The Evans function $D(\lambda)$ can then be defined as \cite{Pego}
\begin{equation}\label{e:Evans}
D(\lambda)=X_+(0;\lambda)\cdot Y_-(0;\lambda).
\end{equation}
While the above construction defines the Evans function for $\realpart{\lambda}>0$, it is well known
that $D(\lambda)$ will continue to be well-defined and complex analytic in any region of the complex plane 
where
$r_1(\lambda)$ continues to remain the unique eigenvalue of  $A^{\infty}(\lambda)$ with
the smallest real part: see \cite{Pego}.  In the region where $r_1(\lambda)$ is the unique eigenvalue of $A^\infty(\lambda)$
with negative real part (which, in our case, includes the entire open right half complex plane), the zeroes of the Evans function
corresponds precisely to the point spectrum of \eqref{DPLeig} considered on $L^2(\RM)$.  For more details, see \cite{Pego}.

While the above discussion uses the Evans function to characterize the point spectrum of $\mathcal{A}[u_0]$ on $L^2(\RM)$,
our goal here is to consider its point spectrum on the weighted space  $L^2_\alpha(\RM)$.  
Proceeding as in Section \ref{s42}, we see that $\lambda\in\CM$ is an  $L^2_\alpha(\RM)$-eigenvalue
of the operator $\mathcal{A}[u_0]$ if and only if $\lambda$ is and $L^2(\RM)$-eigenvalue of the conjugated operator $\mathcal{A}_\alpha[u_0]$
defined in \eqref{spec_w}, whose spectral problem is given explicitly in \eqref{DPLeigessW}.
Similar to the above, the asymptotic behavior of solutions of \eqref{DPLeigessW} is obtained by studying the limiting system where
the coefficients involving $u_0$ are replaced by their (constant) asymptotic values as $\xi\to\infty$.  
Solutions of this constant coefficient system are of the form $e^{r\xi}$ where the $r\in\CM$ are roots of the equation 
\begin{equation} \label{DPLeigessPoly}
\mathcal{P}(\lambda,r-\alpha)=(\lambda+(r-\alpha)(k-c))(1-(r-\alpha)^2)+3k(r-\alpha) = 0, 
\end{equation}
where we note $\mathcal{P}$ is the same as that given in \eqref{Pdef}.  The next result essentially states that in 
the region $\Omega_+(\alpha)$ the Evans function associated to $\mathcal{A}_\alpha[u_0]$ agrees exactly with the 
Evans function for $\mathcal{A}[u_0]$ defined above\footnote{Compare, for instance, to Proposition 2.6 in \cite{Pego}.}.

\begin{lemma}\label{DA}
Let $\alpha>0$ be such that  \eqref{conddef1} holds.  Then:
\begin{itemize}
\item[(i)] The Evans function $D(\lambda)$ defined in \eqref{e:Evans} is well-defined and complex analytic
in $\Omega_+(\alpha)$.  
\item[(ii)] Suppose that $\realpart{\lambda}\geq 0$.  Then $\lambda$ is an eigenvalue of $\mathcal{A}[u_0]$ if and only $D(\lambda)=0$.
\item[(iii)] For $\lambda\in\Omega_+(\alpha)$, we have that $\lambda$ is an $L^2(\RM)$ eigenvalue of $\mathcal{A}_\alpha[u_0]$
if and only if $D(\lambda)=0$.
\end{itemize}
\end{lemma}

\begin{proof}
As mentioned above, the Evans function $D(\lambda)$ will be well-defined and complex analytic in any region of the complex plane 
where $r_1(\lambda)$ continues to remain the unique eigenvalue of $A^\infty(\lambda)$ with smallest real part.  Part (i) 
thus follows immediately by Corollary \ref{C:roots}.  By the same reasoning, the Evans function $D_\alpha(\lambda)$
associated to $\mathcal{A}_\alpha[u_0]$ is well-defined and complex analytic on $\Omega_+(\alpha)$.  Part (ii)
follows from the fact that since the coefficients of $\mathcal{A}[u_0]$ are real that eigenfunctions 
of $\mathcal{A}[u_0]$ associated with purely imaginary eigenvalues necessarily satisfy
the symmetry $v(\xi) = \overline{v(-\xi)}$.  For details, see \cite[proof of Theorem 3.6]{Pego}.
It remains to show that $D_\alpha(\lambda)$ agrees with $D(\lambda)$ on $\Omega_+(\alpha)$.

To this end, note that for $\lambda\in\Omega_+(\alpha)$ we can construct the Evans function associated to $\mathcal{A}_\alpha[u_0]$ 
by performing the substitution
\[
X(\xi)=\tilde{X}(\xi)e^{\alpha\xi}
\]
in \eqref{linear}, which leads to the system
\begin{equation}
\tilde{X}'=\left(A(\xi,\lambda)-\alpha\mathbb{I}_3\right)\tilde{X}.
\label{linearalpha}
\end{equation}
The system \eqref{linearalpha} is recognized exactly as \eqref{DPLeigessW} written as a first-order system.
Following the construction of the Evans function from the beginning of this section, we note that the Evans function $D_\alpha(\lambda)$
associated to \eqref{linearalpha} is obtained by defining 
\[
\tilde{X}_+(\xi)=X_+(\xi)e^{\alpha\xi},~~~\tilde{Y}_-(\xi)=Y_-(\xi)e^{-\alpha\xi}
\]
where $X_+(\xi)$ and  $Y_-(\xi)$ are the solutions of the unweighted linear system \eqref{linear} and its  adjoint system \eqref{adjoint} defined in 
\eqref{XP} and \eqref{YM}, respectively.  The Evans function $D_\alpha(\lambda)$ is then given by
\begin{equation}\label{Evans}
  D_\alpha(\lambda)=\tilde{X}_+(0)\cdot \tilde{Y}_-(0)=X_+(0)\cdot Y_-(0)=D(\lambda),
\end{equation}
establishing the desired equivalence between the Evans function for \eqref{linear} and the Evans function for \eqref{linearalpha}.
\end{proof}

\begin{remark}
Similar to Remark \ref{R:equiv}, it is important to note that roots of $D(\lambda)$ with negative real part are not necessarily
$L^2(\RM)$ eigenvalues of $\mathcal{A}[u_0]$.  By the above result, such roots in $\Omega_+(\alpha)$ are, however, $L^2(\RM)$ eigenvalues
of $\mathcal{A}_\alpha[u_0]$.
\end{remark}

As mentioned at the beginning of this section, Proposition \ref{P:spec_equiv} is now a direct consequence of Lemma \ref{DA} above.  Consequently,
to study the purely imaginary $L^2(\RM)$ eigenvalues of  $\mathcal{A}_\alpha[u_0]$ it is equivalent to study the purely imaginary $L^2(\RM)$ eigenvalues
of $\mathcal{A}[u_0]$.  This analysis is carried out in Section \ref{S:Lax} and Section \ref{S:complete} below by using the complete integrability of the DP equation.

\

Before moving on, we make the following observation about the point spectrum of $\mathcal{A}[u_0]$ acting on $L^2_\alpha(\RM)$
in the case when $\alpha>1$.  Note, in particular, that the next result implies that $\mathcal{A}[u_0]$ does not have a kernel
on $L^2_\alpha(\RM)$ when $\alpha>1$.  For further discussion of this, see Remark \ref{R:triv_ker} below.

\begin{proposition}\label{P:point_bigalpha}
Let $u_0(\cdot;k,c)$ be a smooth solitary wave solution of \eqref{DP} as constructed in Lemma \ref{L1}.  For each $\alpha>1$ the linear operator
$\mathcal{A}[u_0]$ acting on $L^2_\alpha(\RM)$ has no point spectrum on the right side of its essential spectrum described in \eqref{RI}.
\end{proposition}

\begin{proof}
The idea of the proof is very similar to the first part of the proof of Lemma \ref{DA}.  Indeed, for a given $\lambda\in\CM$ we 
recall that the roots of the weighted characteristic
polynomial \eqref{DPLeigessPoly} are of the form $r=r_j+\alpha$, $j=1,2,3$, where the $r_j$ are the roots of the unweighted characteristic
polynomial $P(\lambda,r)$ defined in \eqref{Pdef}.  Further,  that for each $j$ that the sign of $\realpart{r_j+\alpha}$ is necessarily
constant for all $\lambda$ to the right of the essential spectrum of $\mathcal{A}_\alpha[u_0]$ described in \eqref{RI}.
Owing to the asymptotics in \eqref{3roots}, we see that for ${{\realpart{\lambda}}}>0$ with $|\lambda|$ sufficiently large that
the fact that $\alpha>1$ implies that $\realpart{r_j+\alpha}>0$ for $j=1,2,3$.  
It follows that for all such $\lambda$ that  \eqref{DPLeigessW} has three linearly independent solutions diverging
as $\xi\to\infty$, and hence it cannot\footnote{Indeed, the solutions will behave as $\xi\to\infty$ as $\xi^\ell e^{r\xi}$ for some non-zero integer $\ell$.}
have any solutions in $L^2(\RM)$, thus completing the proof.
\end{proof}

\subsection{Lax Pair \& Squared Eigenfunctions}\label{S:Lax}

In order to rule out the existence of non-zero eigenvalues of $\mathcal{A}[u_0]$ embedded in the essential spectrum,
we rely on the complete integrability of the governing DP equation.  Specifically, we use the existence
of a Lax pair.
Despite being similar in appearance to the Camassa-Holm equation \eqref{CH}, the DP equation \eqref{DP} 
has a significantly different underlying integrability structure with its Lax pair being of third order 
(three-dimensional if written as a first-order  dynamical system) and not self-adjoint.   
In the traveling wave coordinates $(\xi,t)=(x-ct,t)$, the Lax pair is given by 
\begin{equation}\label{DPLax}
\left\{\begin{aligned}
&\phi_{\xi\xi\xi} = \phi_\xi + m{\red{\zeta}}\phi,\,\;\;\;m = u - u_{\xi\xi}\\
&\phi_t = \frac{1}{{\red{\zeta}}}\phi_{\xi\xi} +(c- u)\phi_\xi + u_\xi\phi,
\end{aligned}\right.
\end{equation}
where ${\red{\zeta}}\in\CM\setminus\{0\}$ acts as a (Lax) spectral parameter\footnote{For further information about the differences between the CH and DP integrability structures, 
see the discussion after Equation (1.17) in Section 1 and Section 3 in \cite{Lundmark2022}} and where $u=u(x,t)$ acts
as a potential function.  For more details, see \cite{wanghone,dhh}.  Indeed, one can directly verify that if $\phi$ solves the system  \eqref{DPLax}
then the the compatibility condition $\phi_{\xi\xi\xi t}=\phi_{t\xi\xi\xi}$ implies that $m=(1-\partial_\xi^2)u$ satisfies
\[
m_t+(u-c)m_\xi+3u_\xi m=0,
\] 
which of course is equivalent to \eqref{DP} in the traveling coordinate frame.  Additionally, we will make use of the adjoint version\footnote{Note that the system \eqref{DPLaxA} 
is deemed ``adjoint'' to that given in \eqref{DPLax} as the spectral problem in $x$ corresponds to the operator adjoint to the spectral problem in $x$ in \eqref{DPLax}.}
to the Lax pair \eqref{DPLax} given by
\begin{equation}\label{DPLaxA}
\left\{\begin{aligned}
&\psi_{\xi\xi\xi} = \psi_\xi - m{\red{\zeta}}\psi,\,\;\;\;m = u - u_{\xi\xi}.\\
&\psi_t = -\frac{1}{{\red{\zeta}}}\psi_{\xi\xi} +(c- u)\psi_\xi + u_\xi\psi,
\end{aligned}\right.
\end{equation}
As with \eqref{DPLax}, a direct calculation shows that solutions of the system \eqref{DPLaxA} satisfy the compatibility condition $\psi_{\xi\xi\xi t}=\psi_{t\xi\xi\xi}$ 
if and only if $u$ solves the DP equation \eqref{DP} in the traveling wave coordinates.

It is seemingly well known that for a given completely integrable system there generally exists a connection between quadratic combinations of the eigenfunctions of the Lax pair (and its adjoint) and the eigenfunctions for the linear stability stability problem.  See, for example, the works 
{\cite{AKNS1,AKNS2,Gerdjikov1981,Gerdjikov1986,Gerdjikov1994,N85,S83}} along 
with the work \cite{Constantin2007} for the case of the classical Camassa-Holm equation.
In the next result, we establish (seemingly for the first time) 
this so-called squared eigenfunction connection for the DP equation.

\begin{lemma}\label{L:squared_ef}
Let $u_0(\cdot;k,c)$ be a smooth solitary wave to \eqref{DP} as constructed in Lemma \ref{L1}, and suppose the functions $\phi$ and $\psi$ satisfy
the  Lax pairs \eqref{DPLax} and \eqref{DPLaxA}, respectively, considered with potential $u=u_0$.  Then the function
\begin{equation}\label{SEDP1}
v = \left(\phi\psi-\phi_x\psi_x\right)_x
\end{equation}
solves the linearized DP equation given in \eqref{DPL}, i.e. $v$ solves
\[
v_t=\mathcal{A}[u_0]v
\]
where $\mathcal{A}[u_0]$ is the linearized operator defined in \eqref{e:Adef}.
\end{lemma} 

\begin{proof}
While this can be verified by direct calculation, we provide some details regarding how we found this result.
Inspired by work such as \cite{Kaup2010,Gorder2010,kaup2009inverse} dealing with squared eigenfunctions in the cases of 3 and $N$ dimensional Lax pairs, we
systematically searched for a quadratic combination of solutions to the Lax pair \eqref{DPLax} and its adjoint version.
More precisely, we inserted in the linearization \eqref{DPL}  a general linear combination of all the possible terms of the form $\phi$, $\phi_x$, or $\phi_{xx}$ multiplied by  $\psi$, $\psi_x$, or $\psi_{xx}$, where $\phi$ and $\psi$  are solutions of, respectively,  the Lax pair \eqref{DPLax} and its adjoint \eqref{DPLaxA}. The nine constants of the ansatz were determined (up to a remaining free constant) by requiring the linear combination to be a solution of the linearization.
Performing this calculation produces the quadratic combination in \eqref{SEDP1}, as desired.
\end{proof}

{\red{
\begin{remark}\label{re:comp}
For many completely integrable PDEs, one can establish the so-called
completeness of the squared eigenfunctions.
Roughly speaking, such a result asserts that the spectral information of
the linearized operator can be reconstructed from quadratic combinations
of eigenfunctions associated with the Lax pair and its adjoint.
In favorable cases, this provides a generalized Fourier transform under
which perturbations may be expanded over the squared eigenfunctions.
See, for example, \cite{Constantin2007} for the Camassa--Holm equation.

A complete squared-eigenfunction theory is not presently available for
the DP equation.
If such a theory were known, one would expect it to provide direct
information about the spectral properties of $\mathcal{A}[u_0]$,
including the possible presence of $L^2(\mathbb R)$ eigenfunctions.
In the absence of such a result, we establish instead a form of
asymptotic completeness of the squared eigenfunctions when the
(stability) spectral parameter $\lambda$ belongs to the imaginary axis.

The motivation for this approach comes from the general completeness
theory developed by Gerdjikov and collaborators for the generalized
Zakharov--Shabat (equivalently, Caudrey--Beals--Coifman) spectral
problem \cite{Gerdjikov1981,Gerdjikov1986,Gerdjikov1994},
where completeness relations for squared eigenfunctions and the
associated generalized Fourier transform are established.
Furthermore, the DP spectral problem was rewritten in the framework of a
generalized Zakharov--Shabat system in \cite{Constantin2010}.
However, the analysis in
\cite{Gerdjikov1981,Gerdjikov1986,Gerdjikov1994}
assumes rapidly decaying potentials, in particular
$q(x)\to0$ as $|x|\to\infty$.
The background considered here does not satisfy this assumption, and
therefore the completeness results of
\cite{Gerdjikov1981,Gerdjikov1986,Gerdjikov1994}
cannot be applied directly.
Our results may thus be viewed as providing evidence for the expected
completeness picture in the DP setting under substantially different
asymptotic conditions.
\end{remark}
}}

Next, we relate the squared eigenfunction connection \eqref{SEDP1} to the spectral stability problem \eqref{DPLeig}.
To this end, we first note that the adjoint Lax pair \eqref{DPLaxA} can be obtained from the Lax pair \eqref{DPLax} 
by the change of variables $(\xi,t)\to(-\xi,-t)$.  Consequently, the squared eigenfunction connection \eqref{SEDP1} can be reformulated as
\begin{equation}\label{SEDP}
v(\xi,t) = \left(\phi_1(\xi,t)\phi_2(-\xi,-t) + \phi_{1\xi}(\xi,t)\phi_{2\xi}(-\xi,-t)\right)_\xi.
\end{equation}
where the $\{\phi_i\}_{i=1}^2$ are any two solutions of \eqref{DPLax}.  It is thus sufficient
to only study solutions of the Lax pair \eqref{DPLax}.
Inspired by \cite[Section 5]{D2}, we note that the time-evolution equation \eqref{DPLax}(ii) is an autonomous differential equation in 
time and hence admits solutions of the form
\[
\phi(\xi,t) = f(\xi)e^{rt},
\]
where the function $f$ satisfies the ODE system
\eq{
f''' = f' + {\red{\zeta}} m f,\;\;\;rf = \frac{1}{{\red{\zeta}}}f'' +(c- u_0)f' + u_0'f,\,\;\;\;m = u_0 - u_{0}''.
}{DPLaxTW}
Given that the traveling wave $u_0$ is not known explicitly\footnote{{\red{This is in contrast to equations such as KdV, for which explicit solitary and periodic wave solutions are available, allowing explicit representations of the associated Lax system; see, for example, \cite{D2}. In some cases, the Evans function can also be expressed directly in terms of the scattering data. We refer the reader to Remark~\ref{R:formal} for further discussion.}}} for the DP equation, we are not able to solve the ODE system \eqref{DPLaxTW} explicitly.  
However, noting that  $u_0$ decays exponentially fast to its asymptotic end state as $\xi\to\pm\infty$, it follows from standard linear ODE
theory (see for example \cite[p.~104]{codd}) that there exists solutions to the system \eqref{DPLaxTW} that behave asymptotically
as solutions of the free (i.e. asymptotically constant) system
\eq{
\tilde{f}''' = \tilde{f}' + {\red{\zeta}} k \tilde{f},\;\;\;r\tilde{f} = \frac{1}{{\red{\zeta}}}\tilde{f}'' +(c- k)\tilde{f}'{{,}}
}{DPLaxTW2}
which, in turn, has solutions of the form $\tilde{f}(\xi)=e^{l\xi}$ with $l$ satisfying
\begin{equation}\label{lpo}
l^3-l-k{\red{\zeta}}=0
\end{equation}
and the temporal growth rate $r$ in turn being given by
\eq{
r=l(l/{\red{\zeta}}+c-k).
}{rval}
Using the squared eigenfunction connection in \eqref{SEDP}, it follows that  two distinct roots $l_1$ and $l_2$ of \eqref{lpo} would produce
a solution of the linearized evolution equation \eqref{DPL} of the form
\[
v(\xi,t)= e^{(r_1-r_2)t}\left(f_1(\xi)f_2(-\xi)+f_1'(\xi)f_2'(-\xi)\right)',
\]
where the $r_j$ are determined from the $l_j$ via \eqref{rval}.
By the above considerations, it follows from Lemma \ref{L:squared_ef}
that as long as $l_1\neq l_2$ and $l_1l_2\neq -1$ then the ODE \eqref{DPLeig} will admit a solution that behaves asymptotocally 
as\footnote{Note that when $l_1=l_2$, corresponding to a repeated root of \eqref{lpo}, or $l_1l_2=-1$ then \eqref{AsBe} no longer depicts the leading order asymptotics for $\xi\to\infty$.  However, these conditions only occur when $\lambda=0$, which is outside the scope of our argument.  See Remark \ref{R:formal} for more details.}
\begin{equation}\label{AsBe}
{ v(\xi)\sim (l_1-l_2)(1+l_1l_2)e^{(l_1-l_2)\xi},~~\xi\gg 1}
\end{equation}
with the (stability) spectral parameter $\lambda$ given as
\begin{equation}\label{lambda}
\lambda=r_1-r_2=(l_1-l_2)\left((l_1+l_2)/{\red{\zeta}}+c-k\right).
\end{equation}

The above demonstrates a mapping from solutions of the Lax pair \eqref{DPLax} to solutions of the spectral stability problem \eqref{DPLeig} or, equivalently
to solutions of the ODE $\mathcal{A}[u_0]v=\lambda v$.  Specifically, for a given Lax eigenvalue ${\red{\zeta}}$ one finds three roots $\{l_j\}$ 
of \eqref{lpo}, and hence three temporal growth rates $\{r_j\}$ which, in turn, generates (a-priori) six values\footnote{At least, so long
as no two of the roots satisfy $l_1=l_2$ or $l_1l_2=-1$.} of $\lambda$ through
permutations of the $\{r_j\}$ in \eqref{lambda}.  
While this mapping tells us how stability eigenvalues $\lambda$ may be generated from Lax eigenvalues ${\red{\zeta}}$, 
it is a priori unknown if this construction accounts for all possible eigenvalues $\lambda$.  
In the next section, we will explore which solutions of the spectral stability problem
\eqref{DPLeig} are in turn generated through this mapping, at least when $\lambda\in\RM i$.

\subsection{Absence of Non-Zero Point Spectrum}\label{S:complete}

Equipped with the above preliminaries, we can now rule out the existence of non-zero purely imaginary eigenvalues of $\mathcal{A}[u_0]$.
As a first main step, we establish the following technical result concerning the completeness of the squared eigenfunctions
for purely imaginary values of the spectral parameter $\lambda$.  For motivation, we note that the spectral problem \eqref{e:Adef} is a third-order
linear ordinary differential equation and hence, by elementary ODE theory, has exactly three linearly independent solutions for each fixed $\lambda\in\CM$.
The next result shows that for each non-zero $\lambda\in\RM i$, the squared eigenfunction connection in \eqref{SEDP}
provides a complete set of solutions to the third-order ordinary differential equation \eqref{DPLeig}. 

{\red{As discussed in Remark~\ref{re:comp}, we do not establish a completeness result for the squared eigenfunctions themselves, as is often available for completely integrable systems. Rather, the proposition below shows that every solution of the linearized spectral problem \eqref{DPLeig} can be represented through the squared eigenfunction formula \eqref{SEDP1} applied to suitable solutions of the associated Lax pair. In this sense, the squared eigenfunction construction yields all solutions of \eqref{DPLeig} corresponding to non-zero purely imaginary values of the spectral parameter.}}

\begin{proposition}
\label{L2}
Let $u_0(\cdot;k,c)$ be a smooth solitary wave as constructed in Lemma \ref{L1}.  For each $\lambda\in\mathbb{R}i\setminus\{0\}$ 
the squared eigenfunction connection \eqref{SEDP1} provides three linearly independent solutions of the third-order ordinary differential equation \eqref{DPLeig}.
 \end{proposition}

\begin{proof}
By the above considerations, it suffices to show that given any non-zero 
$\lambda\in\RM i$ there {{exist}} three distinct values for $l_1-l_2$, such that $l_1-l_2\neq 0$ and $l_1l_2\neq -1$,
when the $l_j$ are solutions of \eqref{lpo} with constraint \eqref{lambda}.  If we succeed in this, the three corresponding solutions satisfying \eqref{AsBe}
will necessarily be linearly independent\footnote{In other words, and roughly speaking, the strategy is to find three solutions
and then check their linear independence at $\xi=\infty$.} due to their respective behaviors as $\xi\to\infty$.

To this end, we set $M=l_1-l_2$ and $P=l_1+l_2$ and consider the nonlinear system
\begin{equation}\label{sys}
\left\{\begin{aligned}
& \lambda=M\left(P/{\red{\zeta}}+c-k\right),\\
  &l_i^3-l_i-k{\red{\zeta}}=0, i=1,2,\\
\end{aligned}\right.
\end{equation}
in the variables $M$, $P$, and ${\red{\zeta}}$ where here $k$ and $\lambda$ are treated as parameters and, clearly,
\begin{equation}\label{l_equations}
l_1=\frac{P+M}{2},~~l_2=\frac{P-M}{2}{{.}}
\end{equation}
Below, we show that for any value of $\lambda$ on the imaginary axis and $k\in (0,c/4)$, as specified by Lemma \ref{L1}, the solutions of the
system \eqref{sys} for the variables $P$, $M$, and ${\red{\zeta}}$ include three distinct values of $M$ such that $M\neq 1$. Hence, 
there are  three different behaviors in the limit $\xi\to \infty$ as given in \eqref{AsBe} for the differential equation \eqref{DPLeig}, 
corresponding to three linearly independent solutions. 

First, suppose that $\lambda\in\RM i\setminus\{0\}$ and note from \eqref{sys}(i) that this necessarily implies that $M\neq 0$ in this case.
We first claim that $\lambda\neq M(c-k)$, and hence \eqref{sys}(i) can be solved\footnote{Recall {{from}} \eqref{DPLax} that ${\red{\zeta}}$ is necessarily
non-zero.} for ${\red{\zeta}}$ giving
\eq{
  {\red{\zeta}}=\frac{MP}{\lambda+M(k-c)}=\frac{l_1^2-l_2^2}{\lambda+(l_1-l_2)(k-c)}.
  }{littles}  
To verify the claim, note that if  $\lambda=M(c-k)$ then $P=l_1+l_2=0$
by {{the first equation in \eqref{sys}}}.  Since the three roots of \eqref{lpo} necessarily satisfy
\[
l_1+l_2+l_3=0,
\]
it follows that the third root of \eqref{lpo} is necessarily $l_3=0$.  This, however, is not possible since $k{\red{\zeta}}\neq 0$ and hence we must
have $\lambda\neq M(c-k)$.  In particular, since ${\red{\zeta}}\neq 0$ we see from \eqref{littles} that the assumption that $\lambda\neq 0$ implies
that $M\neq 0$ and that $P\neq 0${{.}}

Continuing, with the assumption that $\lambda\in\RM i\setminus\{0\}$, 
we now substitute \eqref{littles} and \eqref{l_equations} into the cubic polynomial \eqref{lpo}, resulting in the two equations
\begin{equation}\label{2eqns}
\left(\frac{P\pm M}{2}\right)^3-\left(\frac{P\pm M}{2}\right)-k\left(\frac{MP}{\lambda+M(k-c)}\right)=0.
\end{equation}
Subtracting these two equations from each other  yields
\[
M\left(\lambda+M(k-c)\right)\left(3P^2+M^2-4\right)=0
\]
and hence, since $\lambda\neq M(c-k)$ and $M\neq 0$ by assumption, it follows that $P$ satisfies
\begin{equation}\label{PF}
3P^2=4-M^2.
\end{equation}
Similarly, adding the two equations in \eqref{2eqns}, using \eqref{PF} to eliminate $P$ from the resulting equation, and recalling that $P\neq 0$ gives
that $M$ satisfies polynomial equation
 \begin{equation}\label{sys2}
(c-k)M^3 - \lambda M^2 + (4k-c)M + \lambda=0.
\end{equation}
The discriminant for the cubic polynomial in  \eqref{sys2} is readily seen to be given by
\begin{equation}\label{Dis}
\Delta=4 \lambda^{4}+\left(61k^2-8 c^{2}-44 c k\right) \lambda^{2}+4 \left(c -k \right) \left(c -4 k \right)^{3},
\end{equation}
which is strictly positive for $\lambda\in\RM i\setminus\{0\}$ since, from Lemma \ref{L1}, we have $k\in(0,c/4)$ and hence
\[
61k^2-8c^2-44ck<61k^2-23k^2-192k^2<0.
\]
It follows for $\lambda\in\CM\setminus\{0\}$ that  the polynomial \eqref{sys2} has three distinct non-zero
solutions for $M=l_1-l_2$.  Finally, we note that the condition $l_1l_2\neq -1$ necessarily holds since otherwise we would have  
\[
P^2-M^2=-4{{,}}
\]
which, combined with \eqref{PF} yields $M=\pm 2$, which is not a solution of \eqref{sys2} for $\lambda\in\RM i\setminus\{0\}$.  
This completes the proof.
\end{proof}

 From Lemma \ref{L2}, the following result follows immediately.

\begin{corollary}
\label{C1}
Any $L^2(\RM)$-eigenfunction for the problem \eqref{DPLeig} corresponding to a non-zero eigenvalue $\lambda$ on the imaginary axis 
corresponds to $L^2(\RM)$-eigenfunctions of the Lax pair \eqref{DPLax} and the adjoint Lax pair \eqref{DPLaxA} through the relation \eqref{SEDP1}.
\end{corollary}

 With Corollary \ref{C1} in hand, we can now establish our main result for this section.  
As mentioned in the Introduction, it is known from previous works  \cite{LP22smooth,Li2019,Li2024} that the smooth solitary waves are spectrally stable
in $L^2(\RM)$ in the sense that the linearized operator $\mathcal{A}[u_0]$ in \eqref{e:Adef} has no $L^2(\RM)$-spectrum in the open right side 
of the complex plane $(\realpart{\lambda}>0)$.  Since the $L^2(\RM)$-spectrum is symmetric about the imaginary axis (see the discussion directly
above Theorem \ref{T:spec_stab}), it follows that any $L^2(\RM)$-point spectrum of $\mathcal{A}[u_0]$ must necessarily be located on the imaginary 
axis.  Using Corollary \ref{C1} along with the known characterization of the $L^2(\RM)$ point spectrum for the Lax pair \eqref{DPLax},
we now establish that $\mathcal{A}[u_0]$ in fact has no non-zero purely imaginary eigenavlues when acting on $L^2(\RM)$.

 \begin{theorem}\label{Th1}
 Let $u_0(\cdot;k,c)$ be a smooth solitary wave of the DP equation \eqref{DP}, as constructed in Lemma \ref{L1}.
The point spectrum of the operator $\mathcal{A}[u_0]$ acting on $L^2(\RM)$ contains only the origin, that is,
\[
\sigma_p\left(\mathcal{A}[u_0];L^2(\RM)\right)=\{0\}.
\]
\end{theorem}

\begin{proof}
 As a result of spectral stability and of Lemma \ref{L2} and the symmetry of the $L^2(\RM)$-spectrum of $\mathcal{A}[u_0]$ about the imaginary axis,
we have that the point spectrum of the eigenvalue problem \eqref{DPLeig} considered on $L^2(\RM)$ can only be located on the imaginary axis.
Further, by Corollary \ref{C1} we know that any $L^2(\RM)$-eigenfunction can be realized through formula \eqref{SEDP1} from a Lax pair eigenfunction. 
From the work in \cite[Section 5]{Constantin2010} and \cite[Section 3.1]{ConstantinIvanov2016}, 
it is known that the DP one-soliton solution $u_0(\cdot;k,c)$ has a Lax point spectrum on $L^2(\RM)$ consisting 
exactly of two real eigenvalues, each the negative of each other.  Additionally, we note that the $L^2(\RM)$-spectrum for the Lax pair \eqref{DPLax}
is symmetric about the origin: if ${\red{\zeta}}\in\CM\setminus\{0\}$ is an eigenvalue with eigenfunction $\phi(x,t)$ then $-{\red{\zeta}}$ is also an eigenvalue
with eigenfunction $\phi(-x,-t)$.  It follows from \eqref{SEDP} that the two (necessarily real) $L^2(\RM)$-eigenvalues of the Lax pair \eqref{DPLax} generate 
only one linearly independent $L^2(\RM)$-solution of the operator $\mathcal{A}[u_0]$.
Since, by Lemma \ref{L2}, all solutions to \eqref{DPLeig} for $\lambda$ on the imaginary axis are realized by formula \eqref{SEDP1}, and because of Corollary \ref{C1}, 
it follows that the total eigenspace for $\mathcal{A}[u_0]$ corresponding to $\lambda\in\RM i$ can be at most one-dimensional.  Since the spectrum
of $\mathcal{A}[u_0]$ is clearly symmetric about the imaginary axis, it follows that $\lambda=0$ can be the only $L^2(\RM)$-eigenvalue
of $\mathcal{A}[u_0]$, as claimed.
 \end{proof}

\begin{remark}\label{R:formal}
Note that the result of {{Theorem}} \ref{Th1} can formally be inferred from the asymptotic relation \eqref{AsBe}.  Indeed, when \eqref{lpo} has 
two roots satisfying $l_1=l_2$ or $l_1l_2=-1$ then the right hand side of \eqref{AsBe} vanishes, formally suggesting\footnote{Of course, this is not at all
rigorous since one would need to study higher order asymptotics in this case.} that the associated
solution vanishes at infinity and hence may be in $L^2(\RM)$.  Note that if $l_1=l_2$ then {{the first equation in \eqref{sys}}} immediately implies that $\lambda=0$.
Similarly, if $l_1l_1=-1$ then, as above, one must have $M=\pm 2$, which solves \eqref{sys2} only when $\lambda=0$.  All together, we
see for $\lambda\in\RM i$ that the right hand side of \eqref{AsBe} vanishes only when $\lambda=0$, suggesting that $\lambda=0$ can be the only
possible eigenvalues of $\mathcal{A}[u_0]$ on $L^2(\RM)$.
\end{remark}

{\red{
\begin{remark}
At first glance, one might expect Theorem~\ref{Th1} to follow from an Evans function analysis. Indeed, although the solitary wave solution of the DP equation is not available in explicit form, it was shown in \cite{Kapitula2002,Kapitula2007} that the Evans function can be expressed in terms of the associated scattering data, even for generalized Zakharov--Shabat systems of dimension greater than two, as is the case for the DP equation. Moreover, the scattering data corresponding to the DP one-soliton are known; see \cite[Section~5]{Constantin2010} and \cite[Section~3.1]{ConstantinIvanov2016}. However, the Evans function formulae derived in \cite{Kapitula2002,Kapitula2007} rely on the existence of a complete set of squared eigenfunctions. At present, such a completeness result is not available for the DP spectral problem (see Remark~\ref{re:comp}), and consequently the Evans function approach cannot be directly applied to establish Theorem~\ref{Th1}.
\end{remark}
}}

We conclude this section by recalling that Lemma \ref{DA} implies that if $\realpart{\lambda}\geq 0$ then $\lambda$ is an $L^2(\RM)$ eigenvalue
of $\mathcal{A}_\alpha[u_0]$  if and only if $\lambda$ is an $L^2(\RM)$ eigenvalue of $\mathcal{A}[u_0]$.  Consequently, Proposition \ref{L3} and Theorem \ref{Th1}
together  immediately imply the following strong spectral stability result for $\mathcal{A}[u_0]$ on $L^2_\alpha(\RM)$.

\begin{proposition}\label{P:strong_spec}
Let $u_0(\cdot;k,c)$ be a smooth solitary wave of the DP equation \eqref{DP} as constructed in Lemma \ref{L1}.
Given any $\alpha>0$  satisfying \eqref{conddef1}, we have that
\[
\sigma_{L^2(\RM)}\left(\mathcal{A}_\alpha[u_0]\right)\cap\left\{\lambda\in\CM:\realpart{\lambda}\geq 0\right\}=\left\{0\right\}.
\]
\end{proposition}

As a next step in our program, we aim to completely characterize the generalized kernel of 
$\mathcal{A}_\alpha[u_0]$ acting on $L^2(\RM)$.  This is completed in the next section.

\subsection{Generalized Kernel in the space $L^2_\alpha(\RM)$}
\label{s5}

Throughout this section, we will assume that $\alpha>0$ is such that the condition \eqref{conddef1} holds, and hence that the essential
spectrum of $\mathcal{A}_\alpha[u_0]$ acting on $L^2(\RM)$ is necessarily stable by Proposition \ref{L3}.  Further, for such $\alpha$
we know from Proposition \ref{P:strong_spec} that $\lambda=0$ is the unique $L^2(\RM)$ eigenvalue of $\mathcal{A}_\alpha[u_0]$ with non-negative 
real part.  Recalling again that $0\in\Omega_+(\alpha)$ it follows that $\lambda=0$ is an isolated eigenvalue of $\mathcal{A}_\alpha[u_0]$
with finite algebraic multiplicity.  
The goal of this  section will be to show the generalized kernel of $\mathcal{A}_\alpha[u_0]$
is generated only by the continuous Lie-point symmetries of the governing DP equation \eqref{DP}.

As motivation for this result, observe that by differentiating the profile equation \eqref{e:profile1} we find that\footnote{Note that since $k$ and $c$ are independent
parameters the functions $u_0'$ and $\partial_cu_0$ necessarily belong to $L^2(\RM)$ with exponential decay rates at infinity determined by the exponential
decay of $u_0$.}
\begin{equation}\label{e:ker_formal}
\mathcal{A}[u_0]u_0'=0~~~{\rm and}~~~\mathcal{A}[u_0]\partial_cu_0=-u_0',
\end{equation}
and hence we expect $\lambda=0$ to be an eigenvalue of $\mathcal{A}[u_0]$ with algebraic multiplicity at least two with at least a Jordan chain of height one above
the translational mode $u_0'$.  Note, however, that since we are working on the  weighted space we must  take particular 
care with the decay rates of the eigenfunctions to ensure they indeed lie in the $L^2_\alpha(\RM)$.  
To this end, we equivalently study the generalized kernel of the conjugated operator\,\footnote{In particular,
note that $\mathcal{A}_\alpha[u_0]$ is Fredholm with index zero.} $\mathcal{A}_\alpha[u_0]$ defined in \eqref{spec_w}
considered on $L^2(\RM)$.  Our main result in this section is the following characterization of the 
left and right generalized kernels of $\mathcal{A}_\alpha[u_0]$ acting on $L^2(\RM)$.

\begin{proposition}\label{P:gker}
Let $u_0(\cdot;k,c)$ be a smooth solitary wave solution of the DP equation \eqref{DP}, and assume $\alpha>0$
is such that the condition \eqref{conddef1} holds.  The operator $\mathcal{A}_\alpha[u_0]$ acting on $L^2(\RM)$ 
has $\lambda=0$ as an eigenvalue with algebraic multiplicity two and geometric multiplicity one.
In particular, defining the functions
\begin{align*}
\tilde{z}_1(\xi)&= u_0'(\xi)\\
 \tilde{z}_2(\xi)&= \partial_cu_0(z)\\
\tilde{\eta}_1(\xi)&=-\theta_1\left(\int_{-\infty}^\xi (1-\partial_y^2)(4-\partial_y^2)^{-1}\partial_cu_0(y)~dy	\right)+\theta_2(1-\partial_\xi^2)(4-\partial_\xi^2)^{-1}\left(u_0(\xi)-k\right)\\
\tilde{\eta}_2(\xi)&=\theta_1 (1-\partial_\xi^2)(4-\partial_\xi^2)^{-1}\left(u_0(\xi)-k\right),
\end{align*}
where 
\[
\theta_1 = \left(\frac{\partial}{\partial c}Q(u_0(\cdot;k,c))\right)^{-1},~~
	\theta_2=\theta_1^2\left(\int_\RM\partial_cu_0(\xi)\left(\int_{-\infty}^\xi(1-\partial_y^2)(4-\partial_y^2)^{-1}\partial_cu_0(y) dy\right)d\xi\right)
\]
and setting
\[
z_j(\xi)=e^{\alpha\xi}\tilde{z}_j(\xi)~~{\rm and}~~\eta_j(\xi)=e^{-\alpha\xi}\tilde{\eta}_j(\xi),
\]
the functions $\{z_j\}_{j=1}^2$ and $\{\eta_\ell\}_{\ell=1}^2$ provide a bi-orthogonal basis for the generalized
kernel of $\mathcal{A}_\alpha[u_0]$ and $\mathcal{A}_\alpha[u_0]^\dag$, respectively, acting on $L^2(\RM)$.  
In particular, we have $\left<z_j,\eta_\ell\right>=\delta_{j\ell}$
and the $z_j$ and $\eta_\ell$ satisfy the equations
\[
\mathcal{A}_\alpha[u_0]z_1=0,~~~\mathcal{A}_\alpha[u_0]z_2 = -z_1
\]
and
\[
\mathcal{A}^\dag_\alpha[u_0]\eta_1\in{\rm span}\left\{\eta_2\right\}\setminus\{0\},~~~\mathcal{A}^\dag_\alpha[u_0]\eta_2=0.
\]
Consequently, the spectral projection $\Pi$ for $\mathcal{A}_\alpha[u_0]$ associated with the eigenvalue $\lambda=0$ is given by\footnote{Observe that while
the spectral projection $\Pi$ is only well-defined for $\alpha>0$ (and satisfying \eqref{conddef1}), the formula for the spectral projection is independent of $\alpha$.}
\[
\Pi f = \sum_{j=1}^2\left<\eta_j,f\right>_{L^2}z_j
\]
for $f\in L^2(\RM)$.  This projection, along with its complementary projection $(1-\Pi)$, commutes with $\mathcal{A}_\alpha[u_0]$ on the domain of 
$\mathcal{A}_\alpha[u_0]$.
\end{proposition}

\begin{proof}
First, we note from Proposition \ref{L3} that when $\alpha>0$ satisfies \eqref{conddef1} the operator $\mathcal{A}_\alpha[u_0]$ is necessarily Fredholm
with index zero.  Consequently, we see that $\lambda=0$ is either in the resolvent set of $\mathcal{A}_\alpha[u_0]$ or is
 an isolated eigenvalue with finite multiplicity.  
Further, from \eqref{e:ker_formal} we readily see by conjugating by $e^{\alpha\xi}$ that we formally have
\begin{equation}\label{e:gker1}
\mathcal{A}_\alpha[u_0]e^{\alpha\xi}u_0'=0,~~{\rm and}~~\mathcal{A}_\alpha[u_0]e^{\alpha\xi}\partial_cu_0=-e^{\alpha\xi}u_0'.
\end{equation}
Noting that \eqref{conddef1} implies that $e^{\alpha\xi}u_0'$ and $e^{\alpha\xi}\partial_cu_0$ both decay exponentially fast as $\xi\to\pm\infty$,
it follows for such $\alpha$  that $\lambda=0$ is an eigenvalue of $\mathcal{A}_\alpha[u_0]$ of algebraic multiplicity at least two.  
Using the Fredholm Alternative, our aim is to show that these functions constitute the entirety of the generalized kernel for $\mathcal{A}_\alpha[u_0]$.

To this end, we first recall that the Hamiltonian structure of \eqref{DP} implies that the conjugated linearized operator $\mathcal{A}_\alpha[u_0]$
can be decomposed as $\mathcal{A}_\alpha[u_0]=J_\alpha\mathcal{L}_\alpha[u_0]$, where here\footnote{Compare to \eqref{e:Adef} and \eqref{e:JLdef} in the unweighted case.}
\begin{align*}
\mathcal{L}_\alpha[u_0]&=(4-(\partial_\xi-\alpha)^2)^{-1}\left((4-(\partial_\xi-\alpha)^2)(u_0-c)+3c\right),\\
J_\alpha&=(\partial_\xi-\alpha)\left(1-(\partial_\xi-\alpha)^2\right)^{-1}\left(4-(\partial_\xi-\alpha)^2\right).
\end{align*}
In particular, $\mathcal{L}_\alpha[u_0]$ is the composition of the boundedly invertible operator $(4-(\partial_\xi-\alpha)^2)^{-1}$
with the Sturm-Liouville operator
\[
\mathcal{S}_\alpha[u_0]:=(4-(\partial_\xi-\alpha)^2)(u_0-c)+3c.
\]
Since $u_0(\xi)<c$ for all $\xi\in\RM$, standard Sturm-Liouville theory implies that the eigenvalues of $\mathcal{S}_\alpha[u_0]$ are all simple and hence,
since $\mathcal{S}_\alpha[u_0]e^{\alpha\xi}u_0'=0$, that the kernel of $\mathcal{S}_\alpha[u_0]$ is simple and spanned by $e^{\alpha\xi}u_0'$.  
Since the operator $J_\alpha$ has a trivial kernel in $L^2(\RM)$, as is easily confirmed using the Fourier transform,
it follows that the kernel of $\mathcal{A}_\alpha[u_0]$ is simple and spanned precisely by $e^{\alpha\xi}u_0'$.

Continuing, we note from \eqref{e:gker1} that $\mathcal{A}_\alpha[u_0]$ admits a Jordan chain above $e^{\alpha\xi}u_0'$ of length at least one.  Consequently,
$\lambda=0$ is necessarily an eigenvalue with algebraic multiplicity at least two.
To show that the algebraic multiplicity is exactly two, we rely on the Fredholm alternative and hence must identify the kernel of the operator
$\mathcal{A}^\dag_\alpha[u_0] = -\mathcal{L}_{-\alpha}[u_0]J_{-\alpha}$ acting on $L^2(\RM)$.  Using that, as above, the operator $J_{-\alpha}$ has a trivial
kernel on $L^2(\RM)$, we see as above that $\lambda=0$ is an eigenvalue of $\mathcal{A}^\dag_\alpha[u_0]$ with geometric multiplicity exactly one and,
in fact, that\footnote{Here, we are using that the solitary wave has the same exponential decay to its asympttoic end state at both plus and minus infinity.}
\[
{\rm ker}\left(\mathcal{A}^\dag_\alpha[u_0]\right)={\rm span}\left\{\left(J_{-\alpha}\right)^{-1} e^{-\alpha\xi}u_0'\right\}={\rm span}\left\{e^{-\alpha\xi}\tilde{\eta}_2\right\}.
\]
By the Fredholm alterantive, it follows that the equation
\[
\mathcal{A}_\alpha[u_0]f=e^{\alpha\xi}\partial_cu_0
\]
will have a solution in $L^2(\RM)$ if and only if $e^{\alpha\xi}\partial_cu_0$ is orthogonal to the kernel of $\mathcal{A}^\dag_\alpha[u_0]$, i.e. if and only if
\[
\left<e^{-\alpha\xi}\tilde{\eta}_2,e^{\alpha\xi}\partial_cu_0\right>=0.
\]
However, note that the above inner product can be rewritten as
\[
\left<(u_0-k),(1-\partial_x^2)(4-\partial_x^2)^{-1}\partial_cu_0\right>=\frac{\partial}{\partial c}Q\left(u_0(\cdot;k,c)\right),
\]
where $Q$ is the momentum functional (associated with the spatial translation invariance of the DP equation) defined \eqref{e:cons1}.
In particular, by Lemma \ref{L:non-deg} it follows that the above derivative is in fact non-zero, and hence that $e^{\alpha\xi}\partial_cu_0$ is not in the range
of $\mathcal{A}_\alpha[u_0]$.  The stated representation and properties of the projection $\Pi$ now follow trivially, thus completing the proof.
\end{proof}

It is important to note that the functions $\tilde{z}_1(\xi)$, $\tilde{z}_2(\xi)$ and $\tilde{\eta}_2(\xi)$ all exhibit exponential decay as $\xi\to\pm\infty$ and,
in particular, they decay at the same exponential rate as the profile $u_0(\xi)$ decays to its asymptotic endstate $k$.  
Recalling Remark \ref{R:decay2}, it readily follows 
that $z_1$, $z_2$, and $\eta_2$ all belong to $L^2(\RM)$ so long as $\alpha$ satisfies \eqref{conddef1}.
The function $\tilde{\eta}_1$, on the other hand, decays exponentially as $\xi\to-\infty$ but is merely bounded as $\xi\to+\infty$ and hence
does not belong to $L^2(\RM)$, indicating that the generalized kernel of the adjoint of $\mathcal{A}[u_0]$
is not two-dimensional when acting on $L^2(\RM)$.   Nevertheless, for $\alpha>0$ satisfying \eqref{conddef1}  the function $\eta_1$ does
decay exponentially as $\xi\to\pm\infty$, i.e. we do have $\eta_1\in L^2(\RM)$, and hence the generalized kernel
of the adjoint of $\mathcal{A}_\alpha[u_0]$ is two-dimensional for such $\alpha$.  This is an important observation as it allows
us to construct a corresponding spectral projection onto the generalized kernel of $\mathcal{A}_\alpha[u_0]$ acting on $L^2(\RM)$,
while such a construction would not be possible for $\mathcal{A}[u_0]$ acting on $L^2(\RM)$.

\begin{remark}\label{R:triv_ker}
By similar considerations as in the discussion directly above, one can easily see how the kernel of $\mathcal{A}_\alpha[u_0]$ is trivial
when $\alpha>1$.  Indeed, from the above work the only candidate for an element in the kernel of $\mathcal{A}_\alpha[u_0]$
is a constant multiple of $e^{\alpha\xi}u_0$.  From Remark \ref{R:decay}, however, we see that 
\[
e^{\alpha\xi}u_0'(\cdot;k,c)\sim e^{(r_{\pm}+\alpha)\xi}~~{\rm as}~~\xi\to\mp\infty.
\]
Since  $|r_{\pm}|<1$, it is immediate that $e^{\alpha\xi}u_0$ is not in $L^2(\RM)$ when $\alpha>1$, as claimed.  Compare,
for instance, to Proposition \ref{P:point_bigalpha}
 \end{remark}

\section{Semigroup Decay Estimates \& Linear Asymptotic Stability}\label{S:LinDecay}

Now that we have studied the spectrum of the linearized operator $\mathcal{A}[u_0]$ on both $L^2(\RM)$ and $L^2_\alpha(\RM)$, we proceed to upgrade the spectral
stability results on $L^2_\alpha(\RM)$ to linear decay estimates.  That is, we now study the linear evolution problem
\[
v_t=\mathcal{A}[u_0]v
\]
posed on $L^2_\alpha(\RM)$ or, equivalently, the linear evolution problem
\[
w_t=\mathcal{A}_\alpha[u_0]w
\]
posed on $L^2(\RM)$ where, recall, $\mathcal{A}_\alpha[u_0]$ is the weighted (conjugated) linear operator defined in \eqref{spec_w}.  To this end, our main goal
will be to prove that $\mathcal{A}_\alpha[u_0]$ generates a $C_0$-semigroup on $L^2(\RM)$ which exhibits exponential decay on the co-dimension
two subspace orthogonal to its generalized kernel.  Following \cite{Pego94}, we first decompose the operator $\mathcal{A}_\alpha[u_0]$
as
\[
\mathcal{A}_\alpha[u_0]=\mathcal{A}_\alpha^\infty + \left(\mathcal{A}_\alpha[u_0]-\mathcal{A}_\alpha^\infty\right)
\]
where here 
\[
\mathcal{A}_\alpha^\infty:=(\partial_\xi-\alpha)(1-(\partial_\xi-\alpha)^2)^{-1}\left((k-c)(4-(\partial_\xi-\alpha)^2)+3c\right)
\]
denotes the (free) asymptotic linearized operator obtained by replacing the solution $u_0$ in  $\mathcal{A}_\alpha[u_0]$ by its asymptotically constant values $k$.
We first study the free evolution generated by $\mathcal{A}_\alpha^\infty$, obtaining exponential decay of the associated semigroup
on $L^2(\RM)$ as well as appropriate resolvent bounds on $L^2(\RM)$ and $H^1(\RM)$.  Through various estimates we then upgrade
these results to the original linearized operator $\mathcal{A}_\alpha[u_0]$.

\subsection{Free Evolution \& Resolvent Estimates}

Note that since the free evolution operator $\mathcal{A}_\alpha^\infty$ has constant coefficients, it clearly generates a $C_0$ semigroup
on both $L^2(\RM)$ and $H^1(\RM)$ via the Fourier transform which, here, we define as
\[
\hat{f}(z) = \int_\RM e^{-i\xi z} f(\xi)d\xi,~~f\in L^2(\RM).
\]
Since the spectrum of $\mathcal{A}_\alpha^\infty$ on $L^2(\RM)$ is precisely the essential spectrum
of the weighted operator $\mathcal{A}_\alpha[u_0]$, the following result follows directly from Proposition \ref{L3}.

\begin{lemma}\label{L:free_decay}
Let $u_0=u_0(\cdot;k,c)$ be a smooth solitary wave of the DP equation \eqref{DP}, as constructed in Lemma \ref{L1}.
For each integer $n$ and for any $0<\alpha<\sqrt{\frac{c-4k}{c-k}}$ any $0<b<\alpha\left(c-k-\frac{3k}{1-\alpha^2}\right)$, there exist a constant $C=C(\alpha,b)>0$
such that
\[
\left\|\partial_\xi^n e^{\mathcal{A}_\alpha^\infty t}w\right\|_{L^2(\RM)}\leq Ce^{-bt}\|\partial_\xi^n w\|_{L^2(\RM)}
\]
for all $t>0$ and $w\in H^n(\RM)$.
\end{lemma}

\begin{proof}
The proof basically boils down to basic properties of the Fourier transform, along with our study of the essential spectrum of $\mathcal{A}_\alpha[u_0]$ 
acting on $L^2(\RM)$ carried out in Section \ref{s42}.  For completeness, the details are provided below.

Given a $w\in H^n(\RM)$, we can use the Fourier transform to write
\begin{align*}
\partial_\xi^n e^{\mathcal{A}_\alpha^\infty t}w(\xi)=\frac{1}{2\pi}\int_\RM e^{i\xi \sigma}(i\sigma)^n e^{(i\sigma-\alpha)(c-k(4-(i\sigma-\alpha)^2)/(1-(i\sigma-\alpha)^2))t}\widehat{w}(\sigma)d\sigma
\end{align*}
and hence using that the Fourier transform is an isometry of $L^2(\RM)$ we find that
\begin{align*}
4\pi \left\|\partial_\xi^n e^{\mathcal{A}_\alpha^\infty t}w\right\|_{L^2(\RM)}^2&
	=\left\| |\sigma|^ne^{\Re\left[(i\sigma-\alpha)(c-k(4-(i\sigma-\alpha)^2)/(1-(i\sigma-\alpha)^2))\right]t}\widehat{w}(\sigma)\right\|_{L^2(\RM;d\sigma)}^2\\
&=\int_\RM |\sigma|^{2n}e^{-2\alpha\left[c-k+3k(\alpha^2+\sigma^2-1)/(4\sigma^2\alpha^2+(1+\sigma^2-\alpha^2)^2\right]t}|\widehat{w}(\sigma)|^2d\sigma.
\end{align*}
By the analysis in Section \ref{s42} and, specifically, Proposition \ref{L3}, we know that for $0<\alpha<\sqrt{\frac{c-4k}{c-k}}$ that
\[
\sup_{\sigma\in\RM}\left(-2\alpha\left[c-k+\frac{3k(\alpha^2+\sigma^2-1)}{(4\sigma^2\alpha^2+(1+\sigma^2-\alpha^2)^2}\right]\right)=2\alpha\left(c-k-\frac{3k}{1-\alpha^2}\right)
\]
and hence for any $0<b<\alpha\left(c-k-\frac{3k}{1-\alpha^2}\right)$ there exists a constant $C>0$ such that
\[
\left\|\partial_\xi^n e^{\mathcal{A}_\alpha^\infty t}w\right\|_{L^2(\RM)}\leq Ce^{-bt}\|\partial_\xi^n w\|_{L^2(\RM)},
\]
as claimed.
\end{proof}

Equipped with the above free evolution estimates, we now aim to show that the weighted operator $\mathcal{A}_\alpha[u_0]$ generates
a $C_0$-semigroup on $L^2(\RM)$ and that, furthermore, the linear decay estimates  from Lemma \ref{L:free_decay} continue to hold for the semigroup generated
by $\mathcal{A}_\alpha[u_0]$ when restricted to the co-dimension two subspace orthogonal to the generalized kernel of $\mathcal{A}_\alpha[u_0]$.
This will be established through by applying both Hille-Yosida and  Prus's theorem which, in turn, both require us to obtain
appropriate resolvent estimates for the linearized operator $\mathcal{A}_\alpha[u_0]$.  As a first step in this direction, we establish 
the following resolvent bounds on the asymptotic operator $\mathcal{A}_\alpha^\infty$.

\begin{lemma}\label{L:free_resolv_bds}
For each $0<a<\alpha<\sqrt{\frac{c-4k}{c-k}}$ there exists constants $C_1>1$ and $C_2>0$ such that for
 $\lambda\in\overline{\Omega_+(a)}$ with $|\lambda|\geq C_1$ we have
and, further, 
\[
\left\|\partial_\xi^n\left(\lambda I-\mathcal{A}_\alpha^\infty\right)^{-1}\right\|_{L^2(\RM)\to L^2(\RM)}
\leq C_2|\lambda|^{n-2},~~n=0,1.
\]
\end{lemma}

\begin{proof}
Since the operator $\mathcal{A}_\alpha^\infty$ has constant coefficients, we can construct the Green's function explicitly
in order to study the resolvent operator $(\lambda I-\mathcal{A}_\alpha^\infty)^{-1}$.  To this end, 
recall from Corollary \ref{C:roots} that $\overline{\Omega_+(a)}\subset\Omega_+(\alpha)$ and that for
$\lambda\in\Omega_+(\alpha)$ the roots of the characteristic polynomial  $\mathcal{P}(\lambda,r-\alpha)$
will have one root of negative real part and two roots with positive real part.  As such,
for such $\lambda$ the equation 
$\mathcal{A}_\alpha^\infty w=\lambda w$ will have two linearly independent solutions that decay exponentially as $\xi\to-\infty$
and one linearly independent solution that will decay exponentially as $\xi\to+\infty$.  For $\lambda\in\Omega_+(\alpha)$ we can thus write the Green's function
explicitly as
\[
G(y;\lambda)=\left\{\begin{aligned}
								&a_1 e^{(r_1+\alpha)y}, &&y>0\\
								&a_2 e^{(r_2+\alpha)y} + a_3 e^{(r_3+\alpha)y}, &&y<0
								\end{aligned}\right.
\]
where the constants $a_j$ are given via
\[
a_j=\frac{1}{(u_0(0)-c)\prod_{i\neq j}(r_j-r_i)}.
\]
Note, in particular, that the constants $a_j$ are chosen to ensure that
\[
G(0+)=G(0-),~~~G_y(0+)=G_y(0-),~~{\rm and}~~G_{yy}(0+)-G_{yy}(0-)=u_0(0)-c,
\]
which can be readily verified.

By construction we have that $(\lambda I-\mathcal{A}_\alpha^\infty)G(\xi;\lambda)=\delta_0$ and hence
the resolvent operator can be represented via
\[
\partial_\xi^n\left(\lambda I-\mathcal{A}_\alpha^\infty\right)^{-1}w=\partial_\xi^nG\ast w
\]
for a given $w\in L^2(\RM)$ and non-negative integer $n$.  In particular, by Young's inequality we have 
\begin{equation}\label{e:young}
\left\|\partial_\xi^n\left(\lambda I-\mathcal{A}_\alpha^\infty\right)^{-1}w\right\|_{L^2(\RM)}\leq \left\|\partial_y^nG(\cdot;\lambda)\right\|_{L^1(\RM)}\|w\|_{L^2(\RM)}
\end{equation}
and hence it remains to estimate the $L^1$ norm of $\partial_y^nG(\cdot;\lambda)$ for $\lambda\in\overline{\Omega_+(a)}$ with $|\lambda|$ large.
To this end, using the explicit formula derived above one easily finds the (relatively crude) bound
\[
\left\|\partial_y^n G(\cdot;\lambda)\right\|_{L^1(\RM)}\leq\frac{1}{|u_0(0)-c|^2}\sum_{j=1}^3\frac{|r_j+\alpha|^n}{|\Re(r_j)+\alpha|} \prod_{\substack{i=1,2,3\\ i\neq j}}\frac{1}{|r_j-r_i|}
\]
that holds for $n=0,1$.  From Corollary \ref{C:roots}, we know for $\lambda\in\overline{\Omega_+(a)}$ that
$\realpart{r_j}+\alpha\geq \alpha-a>0$ for $j=2,3$, while $\realpart{r_1}\to-\infty$ as $|\lambda|\to\infty$ with $\lambda\in\overline{\Omega_+(a)}$.
Consequently, we see that $\left|\realpart{r_j}+\alpha\right|$ is uniformly bounded away from zero for $\lambda\in\Omega_+(a)$ and $|\lambda|$ sufficiently large.
Since \eqref{3roots} implies that $1/|r_j-r_i|=\mathcal{O}(|\lambda|^{-1})$ while $|r_2|=\mathcal{O}(|\lambda|)$ for such $\lambda$ with $|\lambda|$ sufficiently large, the result now follows.
\end{proof}

\subsection{Resolvent Estimates for $\mathcal{A}_\alpha[u_0]$}

Equipped with the above resolvent bounds for the asymptotic operator $\mathcal{A}_\alpha^\infty$, we now
aim to establish resolvent bounds on the weighted linearized operator $\mathcal{A}_\alpha[u_0]$.
The key observation in this direction is the following algebraic identity:
\begin{align*}
\left(\lambda I-\mathcal{A}_\alpha[u_0]\right)^{-1}&= \left(\lambda I-\mathcal{A}_\alpha^\infty\right)^{-1}
	\left( \left(\lambda I-\mathcal{A}_\alpha[u_0]\right) \left(\lambda I-\mathcal{A}_\alpha^\infty\right)^{-1}\right)^{-1}\\
&=\left(\lambda I-\mathcal{A}_\alpha^\infty\right)^{-1}
	\left( \left(\lambda I-\mathcal{A}_\alpha^\infty+\left(\mathcal{A}_\alpha^\infty-\mathcal{A}_\alpha[u_0]\right)\right) \left(\lambda I-\mathcal{A}_\alpha^\infty\right)^{-1}\right)^{-1}\\
&=\left(\lambda I-\mathcal{A}_\alpha^\infty\right)^{-1}
	\left(I-\left(\mathcal{A}_\alpha[u_0]-\mathcal{A}_\alpha^\infty\right) \left(\lambda I-\mathcal{A}_\alpha^\infty\right)^{-1}\right)^{-1}	
\end{align*}
valid for $\lambda\in\rho(\mathcal{A}_\alpha[u_0])\cap\rho(\mathcal{A}_\alpha^\infty)$.  
In particular, this shows in particular, that if we are able to control 
the operator
\begin{equation}\label{e:Ceqn}
 \left(1-\left(\mathcal{A}_\alpha[u_0]-\mathcal{A}_\alpha^\infty\right)\left(\lambda I-\mathcal{A}_\alpha^\infty\right)^{-1}\right)^{-1}
\end{equation}
for sufficiently large $\lambda$ then the free resolvent bounds from Lemma \ref{L:free_resolv_bds} 
can be upgraded to resolvent bounds for the operator $\mathcal{A}_\alpha[u_0]$.
This is the content of the following result.

\begin{proposition}\label{P:resolv_bds}
For each $0<a<\alpha<\sqrt{(c-4k)/(c-k)}$ there exists constants $C_1>1$, $C_2>0$ such that
for $\lambda\in\overline{\Omega_+(a)}$ with $|\lambda|\geq C_1$ we have $\lambda\in\rho(\mathcal{A}_\alpha[u_0])$ and, further, we have the resolvent bound
\[
\left\|\left(\lambda I-\mathcal{A}_\alpha[u_0]\right)^{-1}\right\|_{L^2(\RM)\to L^2(\RM)}
\leq C_2|\lambda|^{-2}.
\]
\end{proposition}

\begin{proof}
To begin, observe that 
\[
\mathcal{A}_\alpha[u_0]-\mathcal{A}_\alpha^\infty = \left(1-(\partial_\xi-\alpha)^2\right)^{-1}\mathcal{B}_\alpha[u_0]
\]
where $\mathcal{B}_\alpha[u_0]$ is a third order ordinary differential operator.  In particular, for each $s\geq 1$ the map
\[
\mathcal{A}_\alpha[u_0]-\mathcal{A}_\alpha^\infty:H^s(\RM)\to H^{s-1}(\RM)
\]
is clearly continuous.  Noting that $\left(\lambda I-\mathcal{A}_\alpha^\infty\right)^{-1}$ maps $H^s(\RM)$ into $H^{s+1}(\RM)$ continuously, it follows
from the free resolvent bounds in Lemma \ref{L:free_resolv_bds} that for $\lambda\in\overline{\Omega_+(a)}$ with $|\lambda|$ sufficiently large
that 
\begin{equation}\label{key_bd_1}
\begin{aligned}
\left\|\left(\mathcal{A}_\alpha[u_0]-\mathcal{A}_\alpha^\infty\right)\left(\lambda I-\mathcal{A}_\alpha^\infty\right)^{-1}g\right\|_{L^2(\RM)}
		&\leq C\left\|(\lambda I-\mathcal{A}_\alpha^\infty)^{-1}g\right\|_{H^1(\RM)}\\
&\leq C\left(|\lambda|^{-2}+|\lambda|^{-1}\right)\|g\|_{L^2(\RM)}\\
&\leq C|\lambda|^{-1}\|g\|_{L^2(\RM)}
\end{aligned}
\end{equation}
valid for all $g\in L^2(\RM)$.
It follows for such $\lambda$ sufficiently large 
that the operator in \eqref{e:Ceqn} is uniformly bounded from $L^2(\RM)$ into itself\footnote{Indeed, one simply needs to ensure
that the operator $\left(\mathcal{A}_\alpha[u_0]-\mathcal{A}_\alpha^\infty\right)\left(\lambda I-\mathcal{A}_\alpha^\infty\right)^{-1}$ has 
norm less than one as a map from $L^2(\RM)$ into itself.}.  Consequently, from the free resolvent bounds in Lemma \ref{L:free_decay} it 
follows that for such $\lambda$ we have
\begin{align*}
\left\|\left(\lambda I-\mathcal{A}_\alpha[u_0]\right)^{-1}f\right\|_{L^2(\RM)}&\leq C|\lambda|^{-2}\left\| \left(1-\left(\mathcal{A}_\alpha[u_0]-\mathcal{A}_\alpha^\infty\right)\left(\lambda I-\mathcal{A}_\alpha^\infty\right)^{-1}\right)^{-1}f \right\|_{L^2(\RM)}\\
&\leq C|\lambda|^{-2}\|f\|_{L^2(\RM)},
\end{align*}
as claimed.
\end{proof}

\subsection{Existence of a Spectral Gap on $L^2_\alpha(\RM)$}\label{S:gap}

We are now in position to prove the existence of a spectral gap for the operator $\mathcal{A}_\alpha[u_0]$ acting on $L^2(\RM)$.  This result
follows from the large $\lambda$ resolvent estimate in Proposition \ref{P:resolv_bds} as well as basic properties of complex analytic functions.

\begin{proposition}\label{P:gap}
Let $u_0(\cdot;k,c)$ be a smooth solitary wave of the DP equation \eqref{DP} as constructed in Lemma \ref{L1}.  For each $\alpha>0$ 
satisfying \eqref{conddef1} there exists a $\eta=\eta(a)>0$ such that
\[
\sigma_{L^2(\RM)}\left(\mathcal{A}_\alpha[u_0]\right)\cap\left\{\lambda\in\CM:\realpart{\lambda}>-\eta\right\}=\{0\}.
\]
\end{proposition}

\begin{proof}
Fix $\alpha>0$ satisfying \eqref{conddef1} and note from Proposition \ref{L3} that there exists an $\varepsilon>0$
such that
\[
\left\{\lambda\in\CM:\realpart{\lambda}>-\varepsilon\right\}\subset\Omega_+(\alpha).
\]
Further, Proposition \ref{P:resolv_bds} implies that the resolvent operator $\left(\lambda I-\mathcal{A}_\alpha[u_0]\right)^{-1}$ is uniformly
bounded for $\lambda\in\Omega_+(\alpha)$  with $|\lambda|$ sufficiently large.  Thus, if the stated result is false 
then there must exist a sequence of $L^2(\RM)$ eigenvalues $\{\lambda_n\}$
of $\mathcal{A}_\alpha[u_0]$ that converges to some point $\lambda^*$ on the imaginary axis.  Since the spectrum
of $\mathcal{A}_\alpha[u_0]$ is a closed set, it follows from Proposition \ref{P:strong_spec} that $\lambda^*=0$.  From Lemma \ref{DA} it 
would follow that the complex analytic function $D(\lambda)$ would have a sequence of zeroes with a finite accumulation
point in its domain of analyticity.  By the Identity Theorem, this would imply that $D(\lambda)=0$ for all $\lambda\in\Omega_+(\alpha)$.
This of course is a contradiction since $D(\lambda)$ does not vanish for $\lambda\in\Omega_+$ with $|\lambda|$ sufficiently
large (again, owing to Proposition \ref{P:resolv_bds}).  The proof is now complete.
\end{proof}

\begin{remark}
As an alternative proof that such a bounded sequence of eigenvalues $\{\lambda_n\}$ above cannot exist, simply note that
since $\mathcal{A}_\alpha[u_0]$ is Fredholm with index zero for all $\lambda\in\Omega_+(\alpha)$ it follows
that $\lambda=0$ must be an isolated eigenvalue.  This immediately implies that there cannot be a sequence of $L^2(\RM)$ eigenvalues
in $\Omega_+(\alpha)$ converging to the finite point $\lambda=0$.
\end{remark}

\subsection{Linear Exponential Decay in $L^2_\alpha(\RM)$}

Equipped with the resolvent bounds in Proposition \ref{P:resolv_bds}, we are now prepared to study the linear dynamics
generated by the (weighted) linearized operator $\mathcal{A}_\alpha[u_0]$.  Using the Hille-Yosida theorem, we first show that 
$\mathcal{A}_\alpha[u_0]$ is the generator of a $C_0$-semigroup on both $L^2(\RM)$
and on $H^1(\RM)$.  Further, owing to the characterization of the $L^2(\RM)$-spectrum of $\mathcal{A}_\alpha[u_0]$
provided in Section \ref{s42}, we then show that $e^{\mathcal{A}_\alpha[u_0]t}$ exhibits exponential decay (with
rates controlled by the size of the spectral gap $\Delta_\alpha$ in Proposition \ref{L3}), at least when restricted to the co-dimension
two subspace orthogonal to the generalized kernel of $\mathcal{A}_\alpha[u_0]$.  

As a first step, we have the following.

\begin{lemma}
Let $u_0(\cdot;k,c)$ be a smooth solitary  wave as constructed in Lemma \ref{L1}.  
For each $\alpha>0$ satisfying \eqref{conddef1} the operator $\mathcal{A}_\alpha[u_0]$ is the generator of a $C_0$-semigroup
on $L^2(\RM)$.
\end{lemma} 

\begin{proof}
The proof relies on the Hille-Yosida Theorem \cite{ER00}.  
From Proposition \ref{P:gap} it follows that
that the semi-infinite interval $(0,\infty)$ lies in the $L^2(\RM)$-resolvent set of $\mathcal{A}_\alpha[u_0]$, while it follows from Proposition \ref{P:resolv_bds}
that there exists an $M>1$ and a constant $C>0$ such that
\[
\left\|\left(\lambda I-\mathcal{A}_\alpha[u_0]\right)^{-1}\right\|_{L^2(\RM)\to L^2(\RM)} \leq C|\lambda|^{-1}
\]
for all $|\lambda|>M$.  In particular, it follows for real $\lambda>M$ that 
\[
\left\|\left(\lambda I-\mathcal{A}_\alpha[u_0]\right)^{-1}\right\|_{L^2(\RM)\to L^2(\RM)} \leq C(\lambda-M)^{-1}.
\]
The proof now follows by the Hille-Yosida Theorem.
\end{proof}

By above and classical linear semigroup theory, it follows that for each $w_0\in L^2(\RM)$ that the initial value problem (IVP)
\[
\left\{\begin{aligned}
&w_t=\mathcal{A}_\alpha[u_0]w\\
&w(0)=w_0
\end{aligned}\right.
\]
has a unique solution $w(t)=e^{\mathcal{A}_\alpha[u_0]t}w_0\in C_0\left([0,\infty);L^2(\RM)\right)$.  Next, we show that if 
the initial data $w_0$ additionally belongs to ${\rm gker}\left(\mathcal{A}_\alpha[u_0]\right)^\perp$ then 
solution to the solution to the above IVP exhibits exponential decay.  The following result is equivalent to our main result Theorem \ref{T:main}
stated in the introduction.

\begin{proposition}\label{P:lin_decay}
Let $u_0(\cdot;k,c)$ be a smooth solitary wave of the DP equation \eqref{DP}, as constructed in Lemma \ref{L1}, let $\alpha>0$
satisfying \eqref{conddef1} be fixed,
and let  $\eta>0$ be such that
\[
\Re\left(\sigma\left(\mathcal{A}_\alpha[u_0]\right)\setminus\{0\}\right)<-\eta.
\]
Then there exists a constant $C>0$ such that for all $w_0\in  L^2(\RM)\cap {\rm gker}\left(\mathcal{A}_\alpha[u_0]\right)^\perp$ we have
\[
\left\|e^{\mathcal{A}_\alpha[u_0]t}w_0\right\|_{L^2(\RM)}\leq Ce^{-\eta t}\|w_0\|_{L^2(\RM)}
\]
for all $t>0$.
\end{proposition}

Note that the existence of such an $\eta>0$ in Proposition \ref{P:lin_decay} is guaranteed by our analysis in Section \ref{s42} and Section \ref{S:gap}.  
The proof of the above linear decay result is based on the following result of Pr\"uss

\begin{theorem}[\cite{P84}, Corollary 4]\label{T:Pruss}
Let $X$ be a Hilbert space and assume that $B$ is the infinitesimal generator of a $C_0$-smigroup on $X$.  Further, let $Q$ be a finite co-dimension 
spectral projection associated with $B$.  If there exists a $\eta>0$ such that
\[
\sup\left\{\left\|\left(\lambda I-B\right)^{-1}Q\right\|_{X\to X}:{{\realpart{\lambda}}}>-\eta\right\}<\infty
\]
then there exists a constant $C>0$ such that 
\[
\left\|e^{Bt}Qf\right\|_X\leq C e^{-\eta t}\|f\|_X
\]
for all $f\in X$.
\end{theorem}

\begin{proof}[Proof of Proposition \ref{P:lin_decay}]
To begin, let $\Pi$ be the rank-$2$ spectral projection onto ${\rm gker}\left(\mathcal{A}_\alpha[u_0]\right)$, 
and let $Q=I-\Pi$ be the complementary projection.  By Theorem \ref{T:Pruss} it is sufficient to show prove
there exists a constant $M>0$ such that
\[
\left\|\left(\lambda I-\mathcal{A}_\alpha[u_0]\right)^{-1}Qf\right\|_{L^2(\RM)}\leq M\|f\|_{L^2(\RM)}
\]
for all $f\in L^2(\RM)$ and all ${{\realpart{\lambda}}}>-\eta$.  To this end, note by Proposition \ref{P:resolv_bds} that
there exists constants $C_1,M_1>0$ such that
\[
\left\|(\lambda I-\mathcal{A}_\alpha[u_0])^{-1}\right\|_{L^2(\RM)\to L^2(\RM)}\leq M_1~~{\rm for~all}~~{{\realpart{\lambda}}}>-\eta,~~|\lambda|>C_1,
\]
providing a uniform bound on the resolvent operator for large $\lambda$.
It remains to control the resolvent operator on the bounded set 
\[
\mathcal{D}:=\left\{\lambda\in\CM:{{\realpart{\lambda}}}>-\eta,~~|\lambda|< C_1+1\right\}.
\]

To this end, note that Proposition \ref{P:gap} implies that $\lambda=0$ is the only eigenvalue of $\mathcal{A}_\alpha[u_0]$ in the interior of $\mathcal{D}$.
It follows that the operator-valued function $\lambda\mapsto\left(\lambda I-\mathcal{A}_\alpha[u_0]\right)^{-1}Q$ is holomorphic
on the compact set $\overline{\mathcal{D}}$, and hence that there exists a constant  $M_2>0$ such that
\[
\left\|(\lambda I-\mathcal{A}_\alpha[u_0])^{-1}Q\right\|_{L^2(\RM)\to L^2(\RM)}\leq M_2~~{\rm for~all}~~\lambda\in\overline{\mathcal{D}}.
\]
It follows that
\[
\left\|(\lambda I-\mathcal{A}_\alpha[u_0])^{-1}Q\right\|_{L^2(\RM)\to L^2(\RM)}\leq \max\left\{M_1,M_2\right\}~~{\rm for~all}~~{{\realpart{\lambda}}}>-\eta,
\]
as required.
\end{proof}

\section{Towards Nonlinear Asymptotic Stability}\label{S:nonlinear_try}

As mentioned in the introduction, we are currently unable to extend the methodologies above to upgrade Proposition \ref{P:lin_decay} to the nonlinear level.
In this section we will outline the challenges encountered in hopes that it sparks further development towards a nonlinear result.  In this section, we will
not go deep into analytical details, but instead will focus on identifying the underlying issues.

Throughout this section, we will assume that $u_0(\cdot;k,c)$ is a smooth 1-soliton of the DP equation \eqref{DP} as constructed in Lemma \ref{L1},
and we assume that the hypotheses of Proposition \ref{P:lin_decay} (or equivalently, Theorem \ref{T:main}) hold.  From the discussion directly
below the statement of Theorem \ref{T:main}, it is natural to expect that a small perturbation of $u_0$ will evolve into a solution of \eqref{DP} that behaves
for large time approximately as
\[
u(\xi,t) \approx u_0\left(x-(c+\beta)t+\gamma;k,c+\beta\right)+\mathcal{O}\left(e^{-\eta t}\right),
\]
where here the approximation is in the $L^2_\alpha(\RM)$-norm.  That is, in $L^2_\alpha(\RM)$ a small perturbation of $u_0$ should evolve into 
a solitary wave of \eqref{DP} with the same asymptotic end state but slightly different wave speed and global phase.  As such, is natural to decompose
solutions near $u_0$ as
\[
u(x,t) = u_0\left(y(x,t);k,c(t)\right)+v(y(x,t),t)
\]
where
\[
y(x,t) = x-\int_0^t c(s)ds + \gamma(t),
\]
where the modulation functions $c(t)$ and $\gamma(t)$ are chosen appropriately\footnote{A standard choice, for example, is to force orthogonality
of the modulated perturbation $v(y,t)$ to the generalized kernel of $\mathcal{A}_\alpha[u_0]$.}.  The perturbation $v$ would then satisfy
a nonlinear equation of the form
\[
v_t=\mathcal{A}_{\alpha,c(t)}v + \mathcal{N}_\alpha(v),
\]
where here $\mathcal{A}_{\alpha,c(t)}$ denotes the linearization of \eqref{DP} about the modulated wave $u_0\left(y(x,t);k,c(t)\right)$ and $\mathcal{N}_\alpha(v)$
denotes nonlinear terms.  Since the coefficients of 
$\mathcal{A}_{\alpha,c(t)}$ depend on $t$,  we rewrite the above as a linear evolution equation of the form\footnote{Here, for this discussion
we are using that the linear estimates in Proposition \ref{P:lin_decay} could be upgraded to $H^1(\RM)$.}
\begin{equation}\label{e:duhamel1}
v_t = \mathcal{A}_\alpha[u_0]v + \left(\mathcal{A}_{\alpha,c(t)}-\mathcal{A}_\alpha[u_0]\right)v+\mathcal{N}_\alpha(v),
\end{equation}
where now we treat the difference $\mathcal{A}_{\alpha,c(t)}-\mathcal{A}_\alpha[u_0]$ as a nonlinear forcing term.  

Using classical dynamical systems methodology, the goal would be to now rewrite \eqref{e:duhamel1} as
\[
v(t) = e^{\mathcal{A}_\alpha[u_0]t}v(0)+\int_0^t e^{\mathcal{A}_\alpha[u_0](t-s)}\left[\left(\mathcal{A}_{\alpha,c(s)}-\mathcal{A}_\alpha[u_0]\right)v(s)+\mathcal{N}_\alpha(v(s))\right]ds
\]
and to iterate the above representation for $v$ to show that solutions
exist and necessarily decay exponentially (once projected off of the generalized kernel of $\mathcal{A}_\alpha[u_0]$) 
and induce appropriate bounds on the modulation functions $c(t)$ and $\gamma(t)$.  However,
a serious issue arises in that the $uu_{xxx}$ term in \eqref{DP} leads to a loss of derviatves above: indeed, by a quick calculation one can show that attempting to 
measure the right hand side of \eqref{e:duhamel1} in  $H^1(\RM)$ requires, (among other things) control of the perturbation in $H^2(\RM)$.  This loss of regularity
in iterating \eqref{e:duhamel1} is the primary significant difficulty in extending our work to the nonlinear level.

We note that a similar issue was present even back in the seminal work of Pego \& Weinstein \cite{Pego1997}.  In their nonlinear analysis of the generalized KdV equation,
they similarly encountered an iteration scheme were derivatives were lost.  In their case, however, their argument was saved by the establishment 
of a so-called ``linear smoothing estimate'' which provided a bound of the $H^2(\RM)$-norm of $e^{L^{\rm KdV}_{\alpha} t}f$ 
in terms of the $H^1(\RM)$ norm of $f$.  As such, the derivative lost in their iteration scheme could be regained through an appropriate linear estimate, allowing
Pego \& Weinstein to close their iteration and achieve a nonlinear stability result.  

Unfortunately, however, for the DP equation \eqref{DP} such a linear smoothing estimate is seemingly not possible.  
Indeed, a straightforward Fourier analysis shows that
\[
\left\|\partial_\xi^n e^{\mathcal{A}_\alpha^\infty t}f\right\|_{L^2(\RM)} = \int_\RM |\sigma|^{2n}e^{\realpart{\lambda(\sigma)} t}|\hat{f}(\sigma)|^2 	d\sigma
\]
where $\lambda(\sigma)$ is given in \eqref{lmroots}.  In our case, the analysis in Section \ref{S:ess_spec} shows that
\[
\lim_{|\sigma|\to\infty}\realpart{\lambda(\sigma)}=-\alpha(c-k)<0
\]
so that, in particular, the function
\[
|\sigma|^{2n}e^{{\text{Re}}(\lambda(\sigma)) t}
\]
is not uniformly bounded\footnote{Note for the KdV equation, for example, $\realpart{\lambda(\sigma)}=-3\alpha\sigma^2$ and a brief rescaling
argument shows that $\sup_{\sigma\in\RM}\left(|\sigma|^{2n}e^{-3\alpha\sigma^2 t}\right)$ exhibits polynomial decay in time.} 
as a function of $\sigma\in\RM$ for any $t>0$.  Consequently, the fact that the essential spectrum of $\mathcal{A}_\alpha[u_0]$
is asymptotically vertical precludes the ability to establish a linear smoothing estimate akin to  \cite[Theorem 4.2]{Pego1997} for the DP equation \eqref{DP}.

\begin{remark}
There have been previous successful attempts at using Pego \& Weintein's methodology to establish asymptotic stability of solitary waves in cases
where the weighted essential spectrum is asymptotically vertical: see, for instance, \cite{Miller1996,Simpson2008}.  In those cases, however,
the evolution equation studied does not include a nonlinear term in the highest order derivaitve and, as such, the authors are able
to absorb terms a priori leading to a loss of derviatives through a simple time rescaling.  While such a rescaling is possible for the DP equation,
it does not absorb the problematic term coming from the nonlinear dispersive term $uu_{xxx}$.  Consequently, it seems a new technique or insight
is needed in order to establish nonlinear asymptotic stability in the present case.
\end{remark}

While the nonlinear asymptotic stability for DP is still open, we hope that the analysis conducted here serves as a first important step in this direction.
In particular, by combining our analysis above with the monotonicity and rigidity approach from \cite{CLLW24} one may be able to
complete a nonlinear argument.  This will be explored in future work.

\bibliographystyle{abbrv}
\bibliography{DJL_10}

\end{document}